\documentclass{amsart}
\usepackage{amssymb}
\newtheorem{thm}{\bf Theorem}[section]
\newtheorem{cor}[thm]{\bf Corollary}
\newtheorem{lem}[thm]{\bf Lemma}

\newtheorem{defn}[thm]{\bf Definition}
\newtheorem{rem}[thm]{\bf Remark}

\numberwithin{equation}{section}

\newcommand{\be}{\begin{equation}}
\newcommand{\ee}{\end{equation}}

\newcommand{\bee}{\begin{equation*}}
\newcommand{\eee}{\end{equation*}}

\newcommand{\bea}{\begin{eqnarray}}
\newcommand{\eea}{\end{eqnarray}}

\newcommand{\Bea}{\begin{eqnarray*}}
\newcommand{\Eea}{\end{eqnarray*}}

\def\R{{\mathbb R}}
\def\Z{{\mathbb Z}}

\def\ol{\overline}

\newcommand{\wh}{\widehat}

\def\CF{\mathcal {F}}
\def\CH{\mathcal {H}}

\def\CM{\mathcal {M}}

\def\CQ{\mathcal {Q}}
\def\CR{\mathcal {R}}
\def\CS{\mathcal {S}}

\begin{document}

\title[$L^p$-Fourier asymptotics, Hardy-type inequality and fractal measures]
{$L^p$-Fourier asymptotics, Hardy-type inequality and fractal measures}
\author{K. S. Senthil Raani}

\address{Department of Mathematics\\
Indian Institute of Science\\
Bangalore-560 012}
\email{kssenthilraani@gmail.com}

\date{\today}


\begin{abstract}
Suppose $\mu$ is an $\alpha$-dimensional fractal measure for some $0<\alpha<n$. Inspired by the results in \cite{Strichartz}, we discuss the $L^p$-asymptotics of the Fourier transform of $fd\mu$ by estimating bounds of

$$\underset{L\rightarrow\infty}{\liminf}\ \frac{1}{L^k} \int_{|\xi|\leq L}\
|\wh{fd\mu}(\xi)|^pd\xi,$$

 for $f\in L^p(d\mu)$ and $2<p<2n/\alpha$. In a different direction, we prove a Hardy type inequality, that is,

 $$\int\frac{|f(x)|^p}{(\mu(E_x))^{2-p}}d\mu(x)\leq C\ \underset{L\rightarrow\infty}{\liminf} \frac{1}{L^{n-\alpha}} \int_{B_L(0)}         |\wh{fd\mu}(\xi)|^pd\xi$$
 where $1\leq p\leq 2$ and $E_x=E\cap(-\infty,x_1]\times(-\infty,x_2]...(-\infty,x_n]$ for $x=(x_1,...x_n)\in\R^n$ generalizing the one dimensional results in \cite{Hudson}.
\end{abstract}
\maketitle


\section{Introduction}
\setcounter{equation}{0}

One of the basic questions in
harmonic analysis is to study the decay properties of the Fourier
transform of measures or distributions supported on thin sets
(sets of Lebesgue measure zero) in $\mathbb{R}^n$. Let $f\in C_c^{\infty}(\R^n)$ and $d\sigma$ be
the surface measure on the sphere $S^{n-1}\subset\R^n$. Then using the properties of Bessel
functions,
$$|\wh{fd\sigma}(\xi)|\leq\ C\ (1+|\xi|)^{-\frac{n-1}{2}}.$$
It follows that $\wh{fd\sigma}\in L^p(\R^n)$ for all
$p>\frac{2n}{n-1}$. This result can be extended to compactly
supported measures on $(n-1)$-dimensional manifolds with
appropriate assumptions on the curvature. On the other hand, the
results in \cite{AgranovskyNaru} show that $\wh{fd\sigma}\notin
L^p(\R^n)$ for $1\leq p\leq \frac{2n}{n-1}$. Similar results are
known for measures supported in lower dimensional manifolds in
$\R^n$ under appropriate curvature conditions (See page 347-351 in \cite{Stein}). However, the
picture for fractal measures is far from complete.\\

In \cite{Raani}, we related the integrability of the function
and the fractal dimension of the support of its Fourier transform
by proving the following theorem and its sharpness.\\

\begin{thm}\cite{Raani} \label{A}  Let $f\in L^1\cap L^p(\R^n)$
be such that $\wh{f}$ is supported in a set $E\subset\R^n$.
Suppose $E$ is a set of finite $\alpha$-packing measure,
$0<\alpha<n$. Then $f$
is identically zero, provided $p\leq 2n/\alpha$.
\end{thm}

Inspired by results in \cite{Strichartz}, we look for quantitative
estimates for Fourier transform of fractal measures. Let $E$ be a
compact set of finite upper Minkowski's $\alpha$-content. In this paper, we obtain certain quantitative
versions of Theorem \ref{A} by estimating the
$L^p$ norm of the Fourier transform of appropriate fractal measures $\mu$ supported in $E$ over
a ball centered at origin with large radius for $1\leq p <
2n/\alpha$, that is, by obtaining bounds for the
following:
  \bee
    \underset{L\rightarrow\infty}{\limsup} \frac{1}{L^k}     \int_{|\xi|\leq    L}|\wh{fd\mu}(\xi)|^p
    d\xi,\label{limsupBall}
  \eee
where $k$ depends on $\alpha,\ p$ and $n$.\\

If $\mu$ is a compactly supported locally uniformly
$\alpha$-dimensional measure, that is, $\mu(B_r(x))\leq
ar^{\alpha}$ for all $0<r\leq 1$ and some non-zero finite
constants $a$, then in \cite{Strichartz}, Strichartz proved that
there exists constant $C$ independent of $f$ such that
   \be
    \ \underset{L\rightarrow\infty}{\limsup} \frac{1}{L^{n-\alpha}}\int_{|\xi|\leq L} |\wh{fd\mu}(\xi)|^2    d\xi \leq C\|f\|^2_{L^2(d\mu)} .
    \label{Str-loc}
   \ee
   The authors in \cite{Lau} and \cite{LauWang} have generalized (\ref{Str-loc}) for a general class of measures. If a locally uniformly $\alpha$-dimensional measure $\mu$ is supported in a quasi $\alpha$-regular set,
then in \cite{Strichartz} the author proved that there exists a non-zero constant independent of $f$ such that
  \be
   \|f\|^2_{L^2(d\mu)} \leq C\ \underset{L\rightarrow\infty}{\liminf} \frac{1}{L^{n-\alpha}}\int_{|\xi|\leq L} |\wh{fd\mu}(\xi)|^2 d\xi.
   \label{Str-quasi}
   \ee

 Using Holder's inequality, we note from (\ref{Str-quasi}) that if
$f\in L^2(d\mu)$, where $\mu$ is a locally uniformly
$\alpha$-dimensional measure, then for $1\leq p\leq 2$,

\be
   \|f\|^2_{L^2(d\mu)} \geq C\ \underset{L\rightarrow\infty}{\limsup}\frac{1}{L^{n-\alpha p/2}}\int_{|\xi|\leq L} |\wh{fd\mu}(\xi)|^pd\xi.\label{str-holder}
 \ee

Suppose $\mu$ is a finite measure supported on a set $E$ such that $\alpha$-dimensional upper Minkowski content of non-zero $\mu$-measure bounded subsets $S$ of $E$ is non-zero and bounded above by $\mu(S)$ for some $0<\alpha<n$, that is,
  $$\underset{\epsilon\rightarrow 0}{\limsup}|S(\epsilon)|\epsilon^{\alpha-n}\leq C\mu(S).$$
Note that $\mu$ need not be locally uniformly $\alpha$-dimensional measure. We prove (\ref{str-holder}) for $\mu$:\\

\begin{thm}\label{Mythm1} Let $f\in L^2(d\mu)$ be a
positive function. Then for $2\leq
p<2n/\alpha$,
        \be
            \int_{\R^n} {|f(x)|^{2}}d\mu(x)\leq\ C\ \underset{L\rightarrow\infty}{\liminf}\ \Big(\frac{1}{L^{n - \alpha p/2}} \int_{|\xi|\leq L}\
|\wh{fd\mu}(\xi)|^pd\xi\Big)^{2/p}.\label{Mythm1eq}
        \ee
        \end{thm}

Next we prove that the above estimate (\ref{Mythm1eq}) is optimal.
\begin{thm}\label{2}Let $u$ be a tempered
distribution supported in a set $E$ such that $\alpha$-dimensional upper Minkowski content of all non-zero $\mu$-measure bounded subsets $S$ of $E$ is non-zero and bounded above by $\mu(S)$ where $\mu$ is a finite measure supported on $E$. Then for $2\leq p <2n/\alpha$,
    \bee
        \underset{L\rightarrow\infty}{\limsup}\ \frac{1}{L^{n-\frac{\alpha p}{2}}} \int_{|\xi|\leq L} |\wh{u}(\xi)|^p d\xi
        <\infty.
    \eee
Then $u$ is an $L^2$ density $\ u_0\ d\mu$ on $E$
and
    \bee
        \Big(\int_{E}|u_0|^2d\mu\Big)^{p/2} \leq\ C\ \underset{L\rightarrow\infty}{\limsup}\ \frac{1}{L^{n-\frac{\alpha p}{2}}}
        \int_{|\xi|\leq L} |\wh{u}(\xi)|^p d\xi <\infty.
    \eee
\end{thm}

Related type of results for tempered distributions supported in a smooth manifold of dimension $d<n$ can be found in \cite{AgmonHormander}, where the authors proved that if $u$ is a tempered distribution such that $\wh{u}\in L^2_{loc}(\R^n)$,
$$\underset{L\rightarrow\infty}{\limsup}\frac{1}{L^{n-d}}\int_{|\xi|\leq L}|\wh{u}(\xi)|^2d\xi<\infty$$
 and the restriction of $u$ to an open subset $X$ of $\R^n$ is supported by a $C^1$ submanifold $M$ of codimension $k=n-d$, then it
is an $L^2$-density $u_0dS$ on $M$ and
$$\int_M|u_0|^2dS\leq C\underset{L\rightarrow\infty}{\limsup}\frac{1}{L^k}\int_{|\xi|\leq L} |\wh{u}(\xi)|^2d\xi,$$
where $C$ only depends on $n$.\\

In a different direction, we consider the results of Hudson and
Leckband in \cite{Hudson} that gives a Hardy type inequality
for discrete measures. Let $\|u\|_{B^p.a.p}^p
= \underset{L\rightarrow\infty}{\lim}\ L^{-1}\int_{-L}^L |u(x)|^p dx$, $fd\mu_0$ be the
zero-dimensional measure $f(x)=\sum_1^{\infty}\ c_k\delta(x-a_k)$
and let $1<p\leq 2$ where $c_k$
 is a sequence of complex numbers and $a_k$ is a sequence of real
numbers not necessarily increasing. Assume that $u(x)=\wh{fd\mu_0}(x)$. The authors in
\cite{Hudson} proved that if
$c_k^*$ denote the nonincreasing rearrangement of the sequence
$|c_k|$, then
\be
\sum_1^{\infty} \frac{|c_k|^p}{k^{2-p}} \leq \sum_1^{\infty} \frac{|c_k^*|^p}{k^{2-p}} \leq C\
\|u\|^p_{B^p.a.p}.\label{Hudsonzero}
\ee

To extend the above result from zero-dimensional sets to a fractional dimensional set, the authors in \cite{Hudson} defined $\alpha$-coherent set in $\R$: for $0<\alpha<1$, a set $E\subset\R^n$ of finite
$\alpha$-dimensional Hausdorff measure is called $\alpha$-coherent
if for all $x$,
 $$\underset{\epsilon\rightarrow0}{\limsup} |E_x^0(\epsilon)|\epsilon^{\alpha-n} \leq C_E\CH_{\alpha}(E_x^0),$$
where $E_x^0=\{y\in E: y\leq x $ and
$2^{-\alpha}\leq\underset{\delta\rightarrow0}{\limsup}
\frac{\CH_{\alpha}(E\cap(y-\delta,y+\delta))}{\delta^{\alpha}}\leq1\}$
and $E_x^0(\epsilon)$ denotes the $\epsilon$-distance set of
$E_x^0$.\\

Let
$E\subset\R$ be either an $\alpha$-coherent set or a quasi
$\alpha$-regular set of finite $\alpha$-dimensional Hausdorff
measure, for $0<\alpha<1$, and $f\in L^1(d\mu)$, where
$\mu=\CH_{\alpha}|_E$. Then the authors in \cite{Hudson} proved that
\be
\int_E \frac{|f(x)|}{\mu(E_x)}d\mu(x)\leq C\ \underset{L\rightarrow\infty}{\liminf} \frac{1}{L^{1-\alpha}}
\int_{-L}^L |\wh{fd\mu}(\xi)|d\xi,\label{Hudsonfractal}
\ee
where $C$ is a constant independent of
$f$.\\

Using the upper Minkowski content and finding a continuous analogue of the
arguments in \cite{Hudson}, we extend (\ref{Hudsonzero}) and (\ref{Hudsonfractal}) to $\R^n$ and also
generalize (\ref{Hudsonzero}) to any $\alpha$-dimensional measure ($0<\alpha<n$) and $n\geq 1$
with a slight modification in the hypothesis. For $E\subset\R^n$, let $$E_x=E\cap (-\infty,x_1]\times...\times(-\infty,x_n]$$ for$x=(x_1,...x_n)\in\R^n$:\\

\begin{thm}\label{3} Let $0<\alpha< n$. Let $\mu$ be a finite measure supported in the set $E$ such that $\alpha$-dimensional upper Minkowski content of all non zero $\mu$-measure bounded subsets $S$ of $E$ is non-zero and bounded by $\mu(S)$. Let $f\in L^p(d\mu)\ (1\leq p\leq 2)$ be a positive
function. Then there exists a constant $C$ independent of $f$ such
that
     \bee
         \int\frac{|f(x)|^p}{(\mu(E_x))^{2-p}}d\mu(x)\leq C\ \underset{L\rightarrow\infty}{\liminf} \frac{1}{L^{n-\alpha}} \int_{B_L(0)}         |\wh{fd\mu}(\xi)|^pd\xi.
    \eee
\end{thm}

We recollect some definitions and notations in the rest of this section from \cite{Cutler}, \cite{Falconer} and \cite{Mattila}. We study the $L^p$-asymptotics of the Fourier
transform of the fractal measures for $p \geq 2$ in
the Section 2 and for $1\leq p\leq2$, we prove generalized Hardy inequality in the Section 3.\\



Let $X$ be a metric space, $\CF$ a family of subsets of $X$ such that for every $\delta>0$ there are $E_1,E_2,...\in\CF$ such that diameter of $E_k$ is less than or equal to $\delta$ for all $k$ and $X=\cup_kE_k$. For $0<\delta\leq\infty$ and $A\subset X$, we define \textbf{$s$-dimensional Hausdorff measure} as

    \begin{equation*}
       \CH_s(A)=\ \underset{\delta\rightarrow0}{\lim}\ H_{\delta}^s(A),
    \end{equation*}
where
    \begin{equation*}
      \CH^s_{\delta}(A)=\inf\ \big\{\sum_{i=1}^{\infty} d(E_i)^s:A\subset\underset{i}{\cup}E_i, d(E_i)\leq\delta, E_i\in\CF\big\}
    \end{equation*}
and $d(E)$ denotes the diameter of the set $E$.

The \textbf{Hausdorff dimension} of a set $A$ is given by
    \Bea
       \dim_{H}(A)&=& \sup\ \{s:\CH_s(A)>0\}=\sup\ \{s:\CH_s(A)=\infty\}\\
          &=& \inf\ \{t:\CH_t(A)<\infty\}=\inf\ \{t:\CH_t(A)=0\}.
    \Eea

For a non-empty subset $A$ of $\R^n$, let $A(\epsilon)=\{x\in\R^n\ : \underset{y\in A}{\inf}\ |x-y|<\epsilon\}$ denote the closed $\epsilon$-neighborhood of $A$. Some authors call $A(\epsilon)$, the \textbf{$\epsilon$-parallel set of $A$} or \textbf{$\epsilon$-distance set of $A$}. Let $E$ be a non-empty bounded subset of $\R^n$. The \textbf{$\epsilon$-covering number} of $E$ denoted by $N(E,\epsilon)$, is the smallest number of open balls of radius $\epsilon$ needed to cover $E$. The \textbf{$\epsilon$-packing number} of $E$ denoted by $P(E,\epsilon)$ is the largest number of \textbf{disjoint} open balls of radius $\epsilon$ with centres in $E$. The \textbf{$\epsilon$-packing} of $E$ is any collection of disjoint balls $\{B_{r_k}(x_k)\}_k$ with centres $x_k\in E$ and radii satisfying $0<r_k\leq\epsilon/2$. Then we have the following lemma:
\begin{lem}\label{lemPM1}\cite{Mattila}: Fix $\epsilon>0$. Let $A$ be a non-empty bounded subset of $\R^n$ and $|A(\epsilon)|$ denote the Lebesgue measure of $A(\epsilon)$, where $A$ is a non-empty bounded subset of $\R^n$. Then,
   \begin{enumerate}
          \item $N(A,2\epsilon)\leq P(A,\epsilon) \leq N(A,\epsilon/2)$.
          \item $\Omega_nP(A,\epsilon)\epsilon^n\leq |A(\epsilon)|\leq \Omega_nN(A,\epsilon)(2\epsilon)^n$,\\
               where $\Omega_n$ denotes the volume of the unit ball in $\R^n$.
          \item For $0\leq s<\infty,\ P(A,\epsilon/2)\epsilon^{s}\leq P_{\epsilon}^{s}(A)$.
   \end{enumerate}
(See pages 78-79 in \cite{Mattila}.)
\end{lem}
The $s$-dimensional \textbf{upper} and \textbf{lower Minkowski contents} of $A$ are defined by
     \bea
       \CM^{*s}(A) &=& \underset{\epsilon\rightarrow 0}{\limsup}\ (2\epsilon)^{s-n}|A(\epsilon)|,\\
       \CM_*^{s}(A) &=& \underset{\epsilon\rightarrow 0}{\liminf}\ (2\epsilon)^{s-n}|A(\epsilon)|.
     \eea

\begin{defn}\label{defnSelfsimilar}\cite{Hutchinson}:
A \textbf{similitude} $S$ is a map $S:\R^n\rightarrow\R^n$ such that
       $$S(x)=sR(x)+b,\ x\in\R^n$$
for some isometry $R$, $b\in\R^n$ and $0<s<1$. The number $s$ is called contraction ratio or dilation factor of $S$. Let $\CS=\{S_1,...S_m\},\ m\geq2$ be a collection of finite set of similitudes with dilation factors $s_1,...,s_m$ (so that $S_j=s_jR_j+b_j$ where $R_j$ denotes an isometry and $b_j\in\R^n$). We say that a non-empty compact set $K$ is \textbf{invariant} under $\CS$ if
      \bee
            K=\cup_{j=1}^m\ S_jK.
      \eee
$\CS$ satisfies the \textbf{open set condition} if there is a non-empty open set $O$ such that $\cup_{j=1}^m\ S_j(O)\subset O$ and $S_{j}(O)\cap S_{k}(O)=\emptyset$ for $j\neq k$. We call the invariant set $K$ under $\CS$ to be \textbf{self-similar} if with $\alpha=\dim_H(K)$,
      \bee
           \CH_{\alpha}(S_{j_1}(K)\cap S_{j_2}(K))=0\ \ \text{for}\ j_1\neq j_2.
      \eee
\end{defn}

\section{$L^p-$asymptotic properties of fractal measures for $2\leq p < 2n/\alpha$}
\setcounter{equation}{0}

Let $\mu$ denote a fractal measure supported in an $\alpha$-dimensional set $E\subset\R^n$ and $f\in L^q(d\mu)$ ($1\leq q\leq\infty$). Suppose $2< p\leq 2n/\alpha$. In this section, we obtain the bounds for
  \be
    \frac{1}{L^{k}}\int_{|\xi|\leq L}|\wh{fd\mu}(\xi)|^pd\xi\label{eq-asymp-2p2nalpha}
  \ee
for some positive $k$.\\

We start recalling a few results proved in \cite{Strichartz}:

\begin{thm}\cite{Strichartz}\label{Strichartzsup} Let $\mu$ be a locally uniformly $\alpha$-dimensional measure and $f\in L^2(d\mu)$. Then
\bee
\underset{0<t\leq 1}{\sup} t^{(n-\alpha)/2}\int_{\R^n}e^{-t|\xi|^2/2}|\wh{fd\mu}(\xi)|^2d\xi \leq C \|f\|_{L^2(d\mu)}^2.
\eee
\end{thm}

Also, with the use of mean quadratic variation, Lau in \cite{Lau}
investigated the fractal measures by defining a class of complex
valued $\sigma$-finite Borel measures $\mu$ on $\R^n$,
$\mathcal{M}_{\alpha}^{p}$, for $1\leq p<\infty$ with
         \bee
               \|\mu\|_{\mathcal{M}_{\alpha}^{p}} = \underset{0<\delta\leq1}{\sup} \Big(\frac{1}{\delta^{n+\alpha(p-1)}}\int_{\R^n} |\mu(Q_{\delta}(x))|^p\Big)^{1/p}<\infty
          \eee
and
          \bee
                \|\mu\|_{\mathcal{M}_{\alpha}^{\infty}} = \underset{u\in\R^n}{sup} \underset{0<\delta<1}{sup}
                \frac{1}{(2\delta)^{\alpha}}|\mu|(Q_{\delta}(u))<\infty,
           \eee
          where $Q_{\delta}(u)$ denotes the half open cube $\prod_{j=1}^n(x_j-\delta, x_j+\delta]$.
For $1\leq p<\infty,\ 0\leq\alpha<n$, $\mathcal{B}_{\alpha}^{p}$
denotes the set of all locally $p$-th integrable function $f$ in
$\R^n$ such that
          \bee
                \|f\|_{\mathcal{B}_{\alpha}^{p}} = \underset{L\geq1}{sup}\Big(\frac{1}{L^{n-\alpha}} \int_{B_L}|f|^p\Big)^{1/p}<\infty.
           \eee
For $0 \leq \alpha\leq\beta<n$, we have from \cite{LauWang}, $\mathcal{B}_{\beta}^{p} \subseteq \mathcal{B}_{\alpha}^{p} \subseteq \mathcal{B}_{0}^{p} \subseteq L^p(dx/(1+|x|^{n+1}))$. For $\delta>0$, we define the transformation $W_{\delta}$ as
           \bee
                (W_{\delta}f)(x)=\int_{\R^n}f(y)E_{\delta}(y)e^{2\pi ix.y}dy,
           \eee
where $E_{\delta}(y)=\int_{|\xi|\leq \delta}e^{2\pi iy\xi}d\xi = 2\pi (\delta|y|^{-1})^{n/2}J_{n/2}(2\pi\delta|y|)$ and $J_{n/2}$ is the Bessel function of order $n/2$. If $\mu$ is a bounded Borel measure on $\R^n$ and $f=\wh{\mu}$, then for $\delta>0$ and for any ball $B_{\delta}(x)$, $\mu(B_{\delta}(x))=(W_{\delta}f)(x)$ for almost all $x\in\R^n$ with respect to $n$-dimensional Lebesgue measure. Lau studied the asymptotic properties of measures in $\mathcal{M}_{\alpha}^{p}$ in \cite{Lau}:

\begin{thm}\label{LauTheorem}\cite{Lau}
     \begin{enumerate}
            \item Let $1\leq p\leq2$, $1/p+1/p'=1$ and $0\leq\alpha<n$. Suppose $\mu\in\mathcal{M}_{\alpha}^{p}$ then $\hat{\mu}\in\mathcal{B}_{\alpha}^{p'}$ with
                  \bee
                         \underset{L\geq1}{sup}\Big(\frac{1}{L^{n-\alpha}} \int_{B_L}|\hat{\mu}|^{p'}\Big)^{1/p'} = \|\hat{\mu}\|_{\mathcal{B}_{\alpha}^{p'}}\leq C\|\mu\|_{\mathcal{M}_{\alpha}^{p}},
                  \eee
                  for some constant $C$ depending on $\mu$.
            \item Let $\mu$ be a positive $\sigma$-finite Borel measure on $\R^n$ and $f$ be any Borel $\mu$-measurable function on $\R^n$. Let $d\mu_f=fd\mu$. $\mu$ is locally uniformly $\alpha$-dimensional if and only if $\|\mu_f\|_{\CM_{\alpha}^p}\leq C\|f\|_{L^p(d\mu)}$ for all $f\in L^p(d\mu)$, $p>1$ and $C$ is a non-zero constant dependent on $p$.
     \end{enumerate}
\end{thm}

In this paper, we concentrate on obtaining lower bounds for (\ref{eq-asymp-2p2nalpha}) as $L$ grows to infinity. Strichartz proved in \cite{Strichartz} an analogue of Radon-Nikodym theorem for  positive measure with no infinite atoms:

\begin{thm}\label{thmStrRN}\cite{Strichartz}
Let $\mu$ be a measure with no infinite atoms, and let $\nu$ be $\sigma$-finite and absolutely continuous with respect to $\mu$. Then there exists a unique decomposition $\nu=\nu_1+\nu_2$ such that $d\nu_1=\phi d\mu$ for a non-negative measurable function $\phi$ and $\nu_2$ is null with respect to $\mu$, that is, $\nu_2(A)=0$ whenever $\mu(A)<\infty$.
\end{thm}
\begin{rem}\label{locuniabshausdorff} As observed in \cite{Strichartz}, any locally uniformly $\alpha$-dimensional measure $\mu$ can be written as $d\mu=\phi d\CH_{\alpha}+d\nu$ where $\nu$ is null with respect to $\CH_{\alpha}$ and $\phi$ is a non negative measurable function belonging to $L^1(\mathbb{R}^n)$.
\end{rem}

In \cite{Strichartz}, the author studied the asymptotic properties
of locally uniformly $\alpha$-dimensional measures and proved a
Plancherel type theorem:

\begin{thm}\label{thmStrMain}\cite{Strichartz}
    Let $\mu'=\mu+\nu$ be a locally uniformly $\alpha$-dimensional measure on $\R^n$ where $\mu=\CH_{\alpha}|_{E}$ and $\nu$ is null with respect to $\CH_{\alpha}$. If $E$ is quasi regular, then for fixed $y$ and constant $c$ independent of $y$,
          \bee
             c\int_{E}|f|^2d\CH_{\alpha}\leq\underset{L\rightarrow\infty}{\liminf}\ \frac{1}{L^{n-\alpha}}\int_{B_L(y)}|\wh{fd\mu'}|^2.
          \eee
  \end{thm}
These results are analogous to the results proved by Agmon and Hormander in \cite{AgmonHormander} when $\alpha$ is an integer. Applying Holder's inequality to the inequality in the Theorem \ref{thmStrMain}, we obtain:

\begin{cor}\label{corMYstr}
Let $f\in L^2(\mu)$ be supported in a quasi $\alpha$-regular set $E$ of non-zero finite $\alpha$-dimensional Hausdorff
measure ($0<\alpha<n$), where $\mu=\mathcal{H}_{\alpha}|_E$ is a locally uniformly $\alpha$-dimensional measure. Then for $p\geq 2$,
        \bee
              \|f\|_{L^2(\mu)}^p\leq c\ \underset{L\rightarrow\infty}{\limsup}\frac{1}{L^{n-\frac{\alpha p}{2}}}\int_{|\xi|\leq L}|\wh{fd\mu}(\xi)|^pd\xi
        \eee
where $c$ is a non zero finite constant depending on $n,\ \alpha$ and $p$.
\end{cor}
The above results are proved for locally uniformly $\alpha$-dimensional measure. Suppose $\mu$ is a finite measure supported on a sparse set $E$. If $\mu$ is such that the $\alpha$-dimensional upper Minkowski content of $S$ is non-zero and bounded above by $\mu(S)$, for any $S\subset E$ such that $\mu(S)\neq 0$, then $\mu$ might not be a locally uniformly $\alpha$-dimensional measure. We prove an analogue result with this measure $\mu$ to the above corollary for the range $2\leq p<2n/\alpha$.

\begin{thm}\label{thm-cohe-str} Let $\mu$ be a finite measure supported in a sparse set $E$ such that the $\alpha$-dimensional upper Minkowski content of $S$ is non-zero and bounded above by $\mu(S)$, for any $S\subset E$ such that $\mu(S)\neq 0$ and let $f\in L^2(d\mu)$ be a positive function. Then,
        \bee
            \int_{\R^n} {|f(x)|^{2}}d\mu(x)\leq\ C\ \underset{L\rightarrow\infty}{\liminf}\ \frac{1}{L^{n - \alpha}} \int_{B_L(0)}\
|\wh{fd\mu}(\xi)|^2d\xi
        \eee
and
       \bee
            \int_{\R^n} {|f(x)|^{2}}d\mu(x)\leq\ C'\ \underset{L\rightarrow\infty}{\liminf}\ \frac{1}{L^{n - \alpha}}\int_{\R^n} e^{-\frac{|\xi|^2}{2L^2}} |\wh{fd\mu}(\xi)|^2d\xi,
       \eee
where the constants $C$ and $C'$ are independent of $f$.
\end{thm}

\begin{proof}
Since $E$ is compact, without loss of generality we assume that
$E$ is contained in a large cube in the positive quadrant, that
is, there exists smallest positive integer $m$ such that for all
$x=(x_1,...x_n)\in E$, $0<x_j<m$. Let $M=\{x=(x_1,...x_n)\in\R^n:
0\leq x_j\leq m,\ \forall j\}$.

Fix $0<\epsilon<1$. For $k=(k_1,...k_n)$, ($0< k_j\in\Z$), \bee
Q_k = \{x=(x_1,...x_n)\in M:\ (k_j-1)\epsilon<x_j\leq
k_j\epsilon,\ j=1,...n\}. \eee Let $\CQ_0$ be the collection of
all such $Q_k$'s whose intersection with $E$ has non-zero
$\mu$-measure, that is, $\mu(Q_k)\neq 0$. Since $E$ is compact,
there exists finite number of $Q_k$'s in $\CQ_0$. Let
$\tilde{\delta_0}=\min_{Q_k\in\CQ_0}\{\mu(Q_k)\}$. Then $E=\cup
(Q_k\cap E)\cup E'$ where the union is finite and $\mu(E')=0$.

       \bea
          \nonumber \int_E|f(x)|^2d\mu(x) &=& \sum_{Q_k\in\CQ_0}\int_{Q_k}
          |f(x)|^2 d\mu(x)\\
          \nonumber &\leq& 2\sum_{Q_k\in\CQ_0} \int_{Q_k}
          \bigg|f(x)-\frac{1}{\mu(Q_k)}\int_{Q_k}f(y)d\mu(y)\bigg|^2 d\mu(x)\\
          && + 2\sum_{Q_k\in\CQ_0} \frac{1}{\mu(Q_k)} \bigg|\int_{Q_k}
          f(y) d\mu(y)\bigg|^2.\label{thmp2packRed1}
          \eea

Now by the hypothesis on $\mu$, for each $k$, there exists $\delta_k$ such that
\bea
\nonumber |(Q_k\cap
E)(\delta)|\delta^{\alpha-n} &\leq& C_n\mu(Q_k\cap E) + C_n\tilde{\delta_0}\epsilon\\
&\leq& 2C_n\mu(Q_k\cap E)= 2C_n\mu(Q_k),
\label{thmp2pack1}
\eea
for all $\delta \leq\delta_k$. Fix $\delta_0= \min\{\epsilon,\ \tilde{\delta_0},\ \delta_1,\ \delta_2,..\}$. Since there are finite $Q_k$'s, $\delta_0>0$. Let $\phi$
be a positive Schwartz function such that $\wh{\phi}(0)=1$,
support of $\wh{\phi}$ is supported in the unit ball and there
exists $r_1>0$ such that
\be
\int_{A_{r_1}(0)} \phi(x) dx =
\frac{1}{2^{n+1}},\label{thmp2pack2}
\ee
where $A_{r_1}(0)=\{x=(x_1,...x_n): -r_1<x_j\leq 0,\ \forall\ j \}$. Denote $\phi_L(x)=\phi(Lx)$ for all
$L>0$. Let $r=n^{\frac{1}{2}}r_1$. Fix $L$ large such that $r/L\leq \delta_0$. Then we have,
      \bea
      \nonumber \bigg|\int_{Q_k} f(y) d\mu(y)\bigg|^2 &=& 2^{2(n+1)} \bigg|\int_{Q_k} \int_{A_{r_1}(0)} \phi(x) dx f(y) d\mu(y)\bigg|^2\\
\nonumber      &=& 2^{2(n+1)}L^{2n}\bigg| \int_{Q_k} \int_{A_{r_1/L}(y)} \phi_L(x-y) dx f(y) d\mu(y)\bigg|^2\\
      &=& 2^{2(n+1)}L^{2n} \bigg|\int_{Q_k^L} \int_{Q_k} \phi_L(x-y)f(y) d\mu(y)dx\bigg|^2,\label{eqthmp2pack0}
      \eea
where $$Q_k^{L}=\{x=(x_1,..x_n)\in\R^n: \exists\ y=(y_1,...y_n)\in
E\ \text{such}\ \text{that}\ y_j-r_1/L<x_j\leq y_j,\ \forall\
j\}.$$ Then $|Q_k^{L}|\leq |(Q_k\cap E)(r/L)|$, where $(Q_k\cap
E)(r/L)$ denotes the $r/L$-distance set of $Q_k\cap E$ (since
$r=\surd{n}r_1$). Also since $\phi$ and $f$ are positive,
$$\int_{Q_k^L} \int_{Q_k} \phi_L(x-y)f(y) d\mu(y)dx\leq
\int_{Q_k^L}\phi_L*fd\mu(x)dx.$$ Thus from (\ref{eqthmp2pack0}),
     \Bea
     \frac{1}{2^{2(n+1)}}\bigg|\int_{Q_k} f(y) d\mu(y)\bigg|^2 &\leq&
      L^{2n}\ \bigg|\int_{Q_k^L}\phi_L*fd\mu(x)dx\bigg|^2.\\
      &\leq& L^{2n} |Q_k^L|\int_{Q_k^L} |\phi_L*fd\mu(x)|^2dx\\
      &\leq& L^{2n}|(Q_k\cap E)(r/L)|\ \int_{Q_k^L} |\phi_L*fd\mu(x)|^2dx\\
      &\leq& 2C_nr^{n-\alpha}L^{n+\alpha}\mu(Q_k)\int_{Q_k^L} |\phi_L*fd\mu(x)|^2dx\ \text{by}\
      (\ref{thmp2pack1}).
      \Eea
Thus there exists a constant $C_1$ independent of $\epsilon$, $L$ and $f$ such that
      \bee
       \frac{1}{\mu(Q_k)}\bigg|\int_{Q_k} f(y) d\mu(y)\bigg|^2 \leq C_1 L^{n+\alpha} \int_{Q_k^L}
       |\phi_L*fd\mu(x)|^2 dx.
      \eee
Hence, from (\ref{thmp2packRed1})
    \Bea
     \int_E|f(x)|^2d\mu(x)&\leq& 2\sum_k\int_{Q_k} \bigg|f(x)-\frac{1}{\mu(Q_k)}\int_{Q_k}f(y)d\mu(y)\bigg|^2 d\mu(x)\\
     && + 2C_1L^{n+\alpha} \sum_{Q_k\in\CQ} \int_{Q_k^L}
       |\phi_L*fd\mu(x)|^2 dx.
     \Eea
By the choice of $r/L<\delta_0<\epsilon$, any $x\in Q_k^L$ belongs to at most $2^n$ number of other $Q_k^L$'s in $\CQ_0$.
Hence there exists a constant $C=2C_12^n$ independent of $f,\epsilon$ and
$L$ such that for all $r/L\leq \delta_0$
 \be
 \int_E|f(x)|^2d\mu(x) \leq 2e_{\epsilon} + CL^{n+\alpha}\int_{E(r/L)}
 |\phi_L*fd\mu(x)|^2dx,\label{eqfin1}
 \ee
 where $$e_{\epsilon}=\sum_{k}\int_{Q_k} |f(x)-\frac{1}{\mu(Q_k)}\int_{Q_k} f(y)d\mu(y)|^2d\mu(x).$$
For given $\epsilon$, let $g\in C_c^{\infty}(d\mu)$ be such that $\|f-g\|^2_{L^2(d\mu)}<\epsilon$. Then,
\Bea
 e_{\epsilon}&=&\sum_{k}\int_{Q_k}\bigg|f(x)-\frac{1}{\mu(Q_k)}\int_{Q_k} f(y)d\mu(y)\bigg|^2d\mu(x)\\
&\leq& 2\sum_{k}\int_{Q_k}\bigg|f(x)-g(x)-\frac{1}{\mu(Q_k)}\int_{Q_k} f(y)-g(y)d\mu(y)\bigg|^2d\mu(x)\\
&& +2\sum_{k}\int_{Q_k}\bigg|g(x)-\frac{1}{\mu(Q_k)}\int_{Q_k} g(y)d\mu(y)\bigg|^2d\mu(x)\\
&\leq& 4\sum_{k}\int_{Q_k}|f(x)-g(x)|^2+\bigg|\frac{1}{\mu(Q_k)}\int_{Q_k} f(y)-g(y)d\mu(y)\bigg|^2d\mu(x)\\
&&+2\sum_{k}\int_{Q_k}\bigg|g(x)-\frac{1}{\mu(Q_k)}\int_{Q_k} g(y)d\mu(y)\bigg|^2d\mu(x)\\
&\leq&8\sum_{k}\int_{Q_k}|f(x)-g(x)|^2d\mu(x)\\
&&+2\sum_{k}\int_{Q_k}\bigg|g(x)-\frac{1}{\mu(Q_k)}\int_{Q_k} g(y)d\mu(y)\bigg|^2d\mu(x).
\Eea
Since $E=\cup_k(Q_k\cap E)\cup E'$,
\bea
\nonumber e_{\epsilon}&\leq& 8\|f-g\|^2_{L^2(d\mu)} +2\sum_{k}\int_{Q_k}\bigg|g(x)-\frac{1}{\mu(Q_k)}\int_{Q_k} g(y)d\mu(y)\bigg|^2d\mu(x)\\
&\leq& \epsilon +2\sum_{k}\int_{Q_k}\bigg|g(x)-\frac{1}{\mu(Q_k)}\int_{Q_k} g(y)d\mu(y)\bigg|^2d\mu(x).\label{Maxop}
\eea
Since $g$ is compactly supported continuous function, $g$ is uniformly continuous and $$|g(x)-\frac{1}{\mu(Q_k)}\int_{Q_k} g(y)d\mu(y)|\rightarrow
0$$ uniformly in $x$ and $Q_k$ as $\mu(Q_k)\rightarrow 0$. As $\epsilon\rightarrow0$, we have $\mu(Q_k)\rightarrow 0$. Hence
\Bea
&&\sum_{k}\int_{Q_k}\bigg|g(x)-\frac{1}{\mu(Q_k)}\int_{Q_k} g(y)d\mu(y)\bigg|^2d\mu(x)\\
&\leq& \sum_{k}\mu(Q_k)\underset{Q_k\in\CQ_0}{\sup}\underset{x\in Q_k}{\sup}\bigg|g(x)-\frac{1}{\mu(Q_k)}\int_{Q_k} g(y)d\mu(y)\bigg|^2\\
&=& \mu(E)\underset{Q_k\in\CQ_0}{\sup}\underset{x\in Q_k}{\sup}\bigg|g(x)-\frac{1}{\mu(Q_k)}\int_{Q_k} g(y)d\mu(y)\bigg|^2,
\Eea
which goes to zero as $\epsilon$ goes to zero. Therefore,  from (\ref{Maxop}), $e_{\epsilon}$ goes to zero as $\epsilon$ goes to zero. Letting $\epsilon$ to $0$, we have $r_1/L\leq\delta_0\rightarrow 0$. Thus (\ref{eqfin1}) becomes
        \be
            \int_E{|f(x)|^2}d\mu(x) \leq\ C\ \underset{L\rightarrow\infty}{\liminf}\ L^{n+\alpha}\int_{E(r/L)}\ {|\phi_L*fd\mu(x)|^{2}}dx,\label{eqredpositive}
\ee
       \Bea
                \int_E|f(x)|^2d\mu(x) &\leq& \ C\ \underset{L\rightarrow\infty}{\liminf}\ L^{n+\alpha}\int_{E(r/L)}\ {|\phi_L*fd\mu(x)|^{2}}dx\\
&\leq&\ C\ \underset{L\rightarrow\infty}{\liminf}\ L^{n+\alpha}\int_{\R^n}\ {|\wh{\phi_L*fd\mu}(\xi)|^{2}}d\xi\\
                        &\leq&\ C\ \underset{L\rightarrow\infty}{\liminf}\ L^{-n+\alpha}\int_{\R^n}\ {|\wh{\phi}(\xi/L)|^{2}|\wh{fd\mu}(\xi)|^{2}}d\xi.
        \Eea
Since the support of $\wh{\phi}$ is in the unit ball, we have
        \bee
              \int_E {|f(x)|^{2}}d\mu(x)\leq\ C\|\phi\|_{L^1(\R^n)}^{2} \underset{L\rightarrow\infty}{\liminf}\ \frac{1}{L^{n - \alpha}} \int_{B_L(0)}\ |\wh{fd\mu}(\xi)|^2d\xi.
         \eee

The assumption on the support of $\widehat{\phi}$ to be in the unit ball is used only in the last step. Consider $\phi(x)=e^{-\frac{|x|^2}{2}}$. Proceeding in a similar way, we have
         $$\int_E {|f(x)|^{2}}d\mu(x)\leq\ C\ \underset{L\rightarrow\infty}{\liminf}\ \frac{1}{L^{n - \alpha}}\int_{\R^n} e^{-\frac{|\xi|^2}{2L^2}} |\wh{fd\mu}(\xi)|^2d\xi$$
Hence the proof.
\end{proof}
By an application of Holder's inequality, we obtain the following Corollary.
\begin{cor}\label{cor-cohe-str} Let $f\in L^2(d\mu)$ be a positive function where $\mu$ is a finite measure supported in a set $E$ such that the $\alpha$-dimensional upper Minkowski content of all non-zero $\mu$-measure subsets $S$ of $E$ is non-zero and bounded above by $\mu(S)$. Then for $2\leq p<2n/\alpha$,
        \bee
            \int_{\R^n} {|f(x)|^{2}}d\mu(x)\leq\ C\ \underset{L\rightarrow\infty}{\liminf}\ \Big(\frac{1}{L^{n - \alpha p/2}} \int_{B_L(0)}\
|\wh{fd\mu}(\xi)|^pd\xi\Big)^{2/p}
        \eee
and
       \bee
            \int_{\R^n} {|f(x)|^{2}}d\mu(x)\leq\ C'\ \underset{L\rightarrow\infty}{\liminf}\ \Big(\frac{1}{L^{n - \alpha p/2}}\int_{\R^n} e^{-\frac{|\xi|^2}{2L^2}} |\wh{fd\mu}(\xi)|^pd\xi\Big)^{2/p},
       \eee
where the constants $C$ and $C'$ are independent of $f$.
\end{cor}

Now, to prove that Theorem \ref{thm-cohe-str} is optimal we prove the following lemma by closely following the
arguments in \cite{AgmonHormander} (also see page 174 of
\cite{Hormander}).

\begin{lem}\label{lemMYhorm1}
Let $u$ be a tempered distribution supported in a compact set $E$. Let $\chi$ be a radial $C_c^{\infty}$ function supported in the unit ball and
$\int_{\mathbb{R}^{n}}\chi(x) dx = 1$. Denote $\chi_{\epsilon}(x)
= \epsilon^{-n}\chi(x/\epsilon)$ and
$u_{\epsilon}=u\ast\chi_{\epsilon}$. Let $\sigma_u(r)=\int_{S^{n-1}} |\widehat{u}(r\omega)|^2d\omega$. Then,
  $$\|u_{\epsilon}\|^{2}\leq C\ \epsilon^{(\alpha-n)(1-\frac{1}{q})} \bigg(\sup_{\epsilon L>1}\frac{1}{L^k} \int_0^{L}(\sigma_u(r))^{\frac{p}{2}}r^{n-1} dr\bigg)^{\frac{2}{p}},$$
for some non-zero finite constants $C$ independent of $\epsilon$ and $k=n-\frac{\alpha p}{2}-(n-\alpha)\frac{p}{2q}$ with $1<q\leq\infty$ and $2\leq p<2n/\alpha$.
\end{lem}

\begin{proof}
By the Plancherel theorem,
  \Bea
    \|u_{\epsilon}\|^{2}
    &=& \int_{\mathbb{R}^{n}} |\widehat{u}(\xi)|^{2}|\wh{\chi}(\epsilon\xi)|^{2} d\xi \\
    &=&\int_0^{\epsilon^{-1}} (\sigma_u(r)) |\wh{\chi}(\epsilon r)|^2 r^{n-1}dr\ +\sum_{j=1}^\infty \int_{2^{j-1}\epsilon^{-1}}^{2^{j}\epsilon^{-1}} (\sigma_u(r)) |\wh{\chi}(\epsilon r)|^2 r^{n-1}dr.\\
    &\leq& \bigg(\int_{0}^{\epsilon^{-1}} (\sigma_u(r))^{p/2}r^{n-1}dr\bigg)^{\frac{2}{p}} \bigg(\epsilon^{-n}\int_{0}^{1} |\hat{\chi}(r)|^{\frac{2}{1-\frac{2}{p}}} r^{n-1} dr \bigg)^{1-\frac{2}{p}}\\
    && + \sum_{j=1}^{\infty} \bigg(\int_{\frac{2^{j-1}}{\epsilon}}^{\frac{2^j}{\epsilon}} \sigma_u(r)^{p/2}r^{n-1}dr\bigg)^{\frac{2}{p}} \bigg(\epsilon^{-n}\int_{2^{j-1}}^{2^j} |\hat{\chi}(r)|^{\frac{2}{1-\frac{2}{p}}} r^{n-1} dr \bigg)^{1-\frac{2}{p}}\\
    &\leq& \epsilon^{(\alpha-n)(1-\frac{1}{q})}\bigg( \sum_{j=0}^{\infty}a_j \bigg(\underset{\epsilon L>1}{\sup} \frac{1}{L^k}\int_0^{L}\sigma_u(r)^{\frac{p}{2}} r^{n-1}dr\bigg)^{\frac{2}{p}}\bigg),
    \Eea
where, for all $j>0$
    \bee
         a_j = \bigg(2^{\frac{2kj}{p-2}} \int_{2^{j-1}}^{2^j}|\wh{\chi}(r)|^{\frac{2p}{p-2}} r^{n-1}dr\bigg)^{1-\frac{2}{p}}
    \eee
and $a_0 = \bigg(\int_{0}^{1}|\wh{\chi}(r)|^{\frac{2p}{p-2}} r^{n-1}dr\bigg)^{1-\frac{2}{p}}.$ We have $\sum_ja_{j}$ is finite. Thus
     \bee
           \|u_{\epsilon}\|^{2} \leq \epsilon^{(\alpha-n)(1-\frac{1}{q})}C \bigg(\underset{\epsilon L>1}{\sup}\frac{1}{L^k} \int_L^{2L}\sigma_u(r)^{\frac{p}{2}} r^{n-1}dr\bigg)^{\frac{2}{p}}
     \eee
  since $k=n-\frac{\alpha p}{2}-(n-\alpha)\frac{p}{2q}$.
\end{proof}

\begin{thm}\label{LBTheorem}
Fix $0<\alpha<n$. Let $\mu$ be a finite measure supported in a compact set $M$ such that the $\alpha$-dimensional upper Minkowski content of $S$ is non-zero and bounded above by $\mu(S)$ for any $S\subset M$ with $\mu(S)\neq 0$. Let $u$ be a tempered distribution such
that support of $u$ is contained in $M$ and
$\sigma_u(r)=\int_{S^{n-1}}|\wh{u}(r\omega)|^2 d\omega$. Let
$2\leq p<\frac{2n}{\alpha}$. Then
         \bee
              \|u\|_1^p\leq C\ \underset{L\rightarrow\infty}{\limsup}\ \frac{1}{L^{n-\frac{\alpha
p}{2}}}\int_0^{L}(\sigma_u(r))^{\frac{p}{2}} r^{n-1}dr \leq C'\ \underset{L\rightarrow\infty}{\limsup}\ \frac{1}{L^{n-\frac{\alpha
p}{2}}}\int_{|\xi|<L}|\widehat{u}(\xi)|^pd\xi,
         \eee
where $\|u\|_1=sup\{<u,\psi>:\psi\in C_c^{\infty}(\R^n),\ \|\psi\|_{L^{\infty}(\R^n)}\leq1\}$, $C$ and $C'$ are non zero finite constants depending only on $n,\ \alpha$ and $p$.\\

In general, for $2\leq p<\frac{2n}{\alpha+\frac{n-\alpha}{q}}$, where $1<q\leq\infty$
         \Bea
              \|u\|_r^p &\leq& C\ \underset{L\rightarrow\infty}{\limsup}\ \frac{1}{L^{n-\frac{\alpha p}{2}-(n-\alpha)\frac{p}{2q}}} \int_0^{L}(\sigma_u(r))^{p}{2} r^{n-1}dr\\
              &\leq& C'\ \underset{L\rightarrow\infty}{\limsup}\ \frac{1}{L^{n-\frac{\alpha
p}{2}-(n-\alpha)\frac{p}{2q}}}\int_{|\xi|<L}|\widehat{u}(\xi)|^pd\xi,
         \Eea
where $\frac{1}{r}+\frac{1}{2q}=1$, $\|u\|_r=sup\{<u,\psi>: \|\psi\|_{L^{2q}(\R^n)}\leq1\}$, $C$ and $C'$ are non zero finite constants depending on $n,\ \alpha$, $p$ and $q$.
\end{thm}

\begin{proof}
Choose an even function $ \chi \in C_{c}^{\infty} (R^{n}) $ with support in unit ball and $\int_{\mathbb{R}^{n}}\chi(x) dx = 1$.
Let $\chi_{\epsilon}(x) = \epsilon^{-n}\chi(x/\epsilon)$ and $u_{\epsilon}=u\ast\chi_{\epsilon}$. Then by Lemma \ref{lemMYhorm1},
      \bee
           \|u_{\epsilon}\|^{2} \leqslant C\ \epsilon^{(\alpha-n)(1-\frac{1}{q})} \bigg(\sup_{\epsilon L>1}\frac{1}{L^k} \int_L^{2L}(\sigma_u(r))^{\frac{p}{2}} r^{n-1} dr\bigg)^{\frac{2}{p}}.
      \eee
Let $\psi\in C_c^{\infty}(\mathbb{R}^n)$. Let $S=supp\ u\cap supp\ \psi$ where $supp\ \psi$ is contained in a ball $B_{R_{\psi}}(0)$ of radius $R_{\psi}$. Since $S$ is a bounded subset of $M$, by hypothesis, we have
      \bee
           \underset{\epsilon\rightarrow0}{\limsup} |S_{\epsilon}| \epsilon^{\alpha-n} \leq c\mu(S) <\infty.
      \eee
For given $0<\delta<1$, there exists $\epsilon_0$ such that for all $\epsilon<\epsilon_0$, $|S(\epsilon)|\epsilon^{\alpha-n}\leq C(\mu(S)+\delta)\leq C_M$. So, for $k=n-\frac{\alpha p}{2}-(n-\alpha)\frac{p}{2q}$,
  \begin{eqnarray*}
      |<u_{\epsilon},\psi>|^2 &\leq& \|u_{\epsilon}\|_2^2 \int_{S_{\epsilon}}|\psi|^2 \\
      &\leq& \|u_{\epsilon}\|_2^2 \bigg(\int_{\R^n}|\psi|^{2q}\bigg)^{\frac{1}{q}} |S_{\epsilon}|^{1-\frac{1}{q}} \\
      &\leq& C_M\|\psi\|_{2q}^2\ \epsilon^{(n-\alpha)(1-\frac{1}{q})}\|u_{\epsilon}\|_2^2\\
      &\leq& C\|\psi\|_{2q}^2 \big(\sup_{\epsilon L>1}\frac{1}{L^k}\int_0^{L}(\sigma_u(r))^{\frac{p}{2}} r^{n-1}dr\big)^{\frac{2}{p}}.
      \end{eqnarray*}
 Thus
       \Bea
              \|u\|_r^p &\leq& C\ \underset{L\rightarrow\infty}{\limsup}\ \frac{1}{L^{n-\frac{\alpha p}{2}-(n-\alpha)\frac{p}{2q}}} \int_0^{L}(\sigma_u(r))^{p}{2} r^{n-1}dr\\
              &\leq& C'\ \underset{L\rightarrow\infty}{\limsup}\ \frac{1}{L^{n-\frac{\alpha p}{2}-(n-\alpha)\frac{p}{2q}}} \int_{|\xi|<L}|\widehat{u}(\xi)|^pd\xi.
       \Eea
\end{proof}
Note that $1\leq r<2$ in the above result. We now prove that the lower bound $\|u\|_r$ in the above theorem can be improved and that Theorem \ref{thm-cohe-str} is optimal.

\begin{thm}\label{thm-alpha-density}
Let $\nu$ be a finite Radon measure supported in a compact set $E$ such that the $\alpha$-dimensional upper Minkowski content of $S$ is non-zero and bounded above by $\nu(S)$ for any $S\subset E$ with $\nu(S)\neq 0$. Let $u$ be a tempered distribution supported in $E$ such that for some $2\leq p <2n/\alpha$,
    \bee
        \underset{L\rightarrow\infty}{\limsup}\ \frac{1}{L^{n-\frac{\alpha p}{2}}} \int_{|\xi|\leq L} |\wh{u}(\xi)|^p d\xi <\infty.
    \eee
Then $u$ is an $L^2$ density $\ u_0\ d\nu$ on $E$ and
    \bee
        \Big(\int_{E}|u_0|^2d\nu\Big)^{p/2} \leq\ C\ \underset{L\rightarrow\infty}{\limsup}\ \frac{1}{L^{n-\frac{\alpha p}{2}}} \int_{|\xi|\leq L} |\wh{u}(\xi)|^p d\xi <\infty.
    \eee
\end{thm}

\begin{proof}
Let $\psi\in C_c^{\infty}(\mathbb{R}^n)$. Let $S=supp\ u\cap supp\ \psi$. Then $S$ is bounded and let $M$ be the smallest closed cube that contains $S$. As in Theorem \ref{thm-cohe-str}, for $0<\delta<1$, let $\tilde{\CQ_0}$ be the collection of all half open cubes $Q_k=\{x=(x_1,...x_n)\in M: (k_j-1)\delta<x_j \leq k_j\delta\}$, ($k=(k_1,...k_n),\ k_j\in\Z$) and $\CQ_0$ be the collection of all $Q_k\in\tilde{\CQ_0}$ such that $\nu(Q_k\cap E)\neq 0$. Denote $\mu=\nu|_S$. Since $S$ is bounded, there are finite $Q_k$'s in $\CQ_0$. Let $\delta_0=\min_{Q_k\in\CQ_0}\{\mu(Q_k)\}$. By hypothesis, for each $k$, there exists $\delta_k$ such that
\bea
\nonumber |(Q_k\cap
S)(\epsilon)|\epsilon^{\alpha-n} &\leq& C_n\mu(Q_k\cap S) + C_n\tilde{\delta_0}\delta\\
&\leq& 2C_n\mu(Q_k\cap S),
\label{thmdensity1}\\
|S(\epsilon)|\epsilon^{\alpha-n}&\leq& \mu(S) + \delta\label{thmdensity3}
\eea
for all $\epsilon \leq\delta_k$. Fix $\epsilon_0= \min\{\delta,\ \delta_0,\ \delta_1,\ \delta_2,..\}$. For every $\epsilon<\epsilon_0$, let $\CQ^{\epsilon}_0$ denote the collection of all $Q_k$ in $\tilde{\CQ_0}$ such that $|Q_k\cap S(\epsilon)|\neq 0$.
\Bea
 \epsilon^{\alpha-n}\int_{S(\epsilon)}|\psi(x)|^2dx &=& \epsilon^{\alpha-n}\sum_{Q_k\in\CQ^{\epsilon}_0}\int_{Q_k\cap S(\epsilon)} |\psi(x)|^2 dx\\
&\leq& \epsilon^{\alpha-n}\sum_{Q_k\in\CQ^{\epsilon}_0\backslash\CQ_0} \int_{Q_k\cap S(\epsilon)} |\psi(x)|^2 dx\\
&& +2\epsilon^{\alpha-n}\sum_{Q_k\in\CQ_0}\int_{Q_k\cap S(\epsilon)} \bigg|\psi(x)-\frac{1}{\mu(Q_k)}\int_{Q_k}\psi(y)d\mu(y)\bigg|^2 dx\\
&& + 2\epsilon^{\alpha-n}\sum_{Q_k\in\CQ_0}\int_{Q_k\cap S(\epsilon)} \bigg|\frac{1}{\mu(Q_k)}\int_{Q_k}\psi(y)d\mu(y)\bigg|^2 dx.
\Eea
Since, for $Q_k\in\CQ^{\epsilon}_0\backslash\CQ_0$, $\mu(Q_k)=0$, from (\ref{thmdensity1}),
\Bea
&&\epsilon^{\alpha-n}\sum_{Q_k\in\CQ^{\epsilon}_0\backslash\CQ_0} \int_{Q_k\cap S(\epsilon)} |\psi(x)|^2 dx\\
&&\ \ \leq 2C_n \|\psi\|_{\infty}^2 \sum_{Q_k\in\CQ^{\epsilon}_0\backslash\CQ} \mu(Q_k) = 0.
\Eea
Hence,
\bea
 \nonumber\epsilon^{\alpha-n}\int_{S(\epsilon)}|\psi(x)|^2dx &\leq& 2\epsilon^{\alpha-n}\sum_{Q_k\in\CQ_0}\int_{Q_k\cap S(\epsilon)} \bigg|\psi(x)-\frac{1}{\mu(Q_k)}\int_{Q_k}\psi(y)d\mu(y)\bigg|^2 dx\\
\nonumber&& + 2\epsilon^{\alpha-n}\sum_{Q_k\in\CQ_0}\int_{Q_k\cap S(\epsilon)} \bigg|\frac{1}{\mu(Q_k)}\int_{Q_k}\psi(y)d\mu(y)\bigg|^2 dx\\
\nonumber &\leq& e_{\delta} +2\sum_{Q_k\in\CQ_0} \epsilon^{\alpha-n}|Q_k\cap S(\epsilon)|\frac{1}{\mu(Q_k)}\int_{Q_k}|\psi(y)|^2d\mu(y),
\eea
where $$e_{\delta}=2\sum_{Q_k\in\CQ_0}\epsilon^{\alpha-n}\int_{Q_k\cap S(\epsilon)} \bigg|\psi(x)-\frac{1}{\mu(Q_k)}\int_{Q_k}\psi(y)d\mu(y)\bigg|^2 dx.$$ By (\ref{thmdensity1}), $$\epsilon^{\alpha-n}|Q_k\cap S(\epsilon)|\leq \epsilon^{\alpha-n} |(Q_k\cap S)(\epsilon)|\leq 2C_n\mu(Q_k).$$ Hence
\bea
\nonumber \epsilon^{\alpha-n}\int_{S(\epsilon)}|\psi(x)|^2dx &\leq& e_{\delta} +4C_n\sum_{Q_k\in\CQ_0} \int_{Q_k}|\psi(y)|^2d\mu(y)\\
&=& e_{\delta} +4C_n\int_{E}|\psi(y)|^2d\mu(y).\label{thmdensity2}
\eea
Since $\psi$ is compactly supported continuous function, $|\psi(x)-\frac{1}{\mu(Q_k)}\int_{Q_k}\psi(y)d\mu(y)|\rightarrow 0$ uniformly in $x$ and $Q_k$ as $\delta$ goes to zero. $\underset{x\in S(\epsilon)}{\sup}|\psi(x)-\frac{1}{\mu(Q_k)}\int_{Q_k}\psi(y)d\mu(y)|\rightarrow 0$ as $\delta$ goes to zero.
\Bea
e_{\delta} &=& \epsilon^{\alpha-n}\sum_{Q_k\in\CQ_0} \int_{Q_k\cap S(\epsilon)} |\psi(x)-\frac{1}{\mu(Q_k)}\int_{Q_k}\psi(y)d\mu(y)|^2 dx\\
&\leq& \epsilon^{\alpha-n}\sum_{Q_k\in\CQ_0}|Q_k\cap S(\epsilon)|\sup_{x\in S(\epsilon)}  \bigg|\psi(x)-\frac{1}{\mu(Q_k)}\int_{Q_k}\psi(y)d\mu(y)\bigg|^2 \\
&\leq& \epsilon^{\alpha-n}|S(\epsilon)| \sup_{x\in S(\epsilon)}  \bigg|\psi(x)-\frac{1}{\mu(Q_k)}\int_{Q_k}\psi(y)d\mu(y)\bigg|^2.
\Eea
Then together with (\ref{thmdensity3}), $e_{\delta}$ goes to zero as $\delta$ goes to zero. Thus from (\ref{thmdensity2}), for given $0<\delta<1$, there exists small $\epsilon_0$ such that for all $\epsilon<\epsilon_0$,
\bea
\nonumber\epsilon^{\alpha-n}\int_{S(\epsilon)}|\psi(x)|^2dx &\leq& e_{\delta} + 4C_n \int_{E}|\psi(y)|^2d\nu(y)\\
&=& e_{\delta} + 4C_n \|\psi\|^2_{L^2(d\nu)}.\label{thmdensityfin1}
\eea
where $e_{\delta}$ tends to zero as $\delta$ tends to zero.\\

Now we proceed as in the Theorem \ref{LBTheorem}. Choose an even function $ \chi \in C_{c}^{\infty} (R^{n}) $ with support in unit ball and $\int_{\mathbb{R}^{n}}\chi(x) dx = 1$. Let $\chi_{\epsilon}(x) = \epsilon^{-n}\chi(x/\epsilon)$ and $u_{\epsilon}=u\ast\chi_{\epsilon}$. Then by Lemma \ref{lemMYhorm1},
       \bea
          \nonumber   \|u_{\epsilon}\|^{2} &\leq& C\ \epsilon^{\alpha-n}\bigg(\sup_{\epsilon L>1}\frac{1}{L^{n-\alpha p/2}} \int_0^{L} (\sigma_u(r))^{\frac{p}{2}} r^{n-1}dr\bigg)^{\frac{2}{p}}\\
             &\leq& C\ \epsilon^{\alpha-n}\bigg(\sup_{\epsilon L>1}\frac{1}{L^{n-\alpha p/2}} \int_{|\xi|<L} |\hat{u}(\xi)|^pd\xi\bigg)^{\frac{2}{p}}.\label{thmdensityfin2}
      \eea
We have $\epsilon\rightarrow 0$ as $\delta\rightarrow 0$. Thus
     \begin{eqnarray*}
         |<u,\psi>|^2 &=& \underset{\epsilon\rightarrow0}{\lim}\ |<u_{\epsilon},\psi>|^2 \\
             &\leq& \underset{\epsilon\rightarrow0}{\lim}\ \|u_{\epsilon}\|_2^2\int_{S_{\epsilon}}|\psi|^2 \\
             &\leq& \underset{\epsilon\rightarrow0}{\lim}\ \|u_{\epsilon}\|_2^2\epsilon^{n-\alpha}(e_\delta + C\|\psi\|_{L^2(d\nu)}^2)\ \text{from}\ (\ref{thmdensityfin1}).
\end{eqnarray*}
Thus letting $\delta$ go to zero, together with (\ref{thmdensityfin2}),
\bee
            |<u,\psi>|^2 \leq C\|\psi\|_{L^2(d\mu)}^2\ \big(\underset{L\rightarrow\infty}{\limsup} \frac{1}{L^{n-\alpha p/2}} \int_{|\xi|\leq L}|\hat{u}(\xi)|^pd\xi\big)^{\frac{2}{p}}
     \eee
 Thus $u$ is an $L^2$ density $\ u_0\ d\nu$ on $E$ and
      \bee
          \Big(\int_{E}|u_0|^2d\nu\Big)^{p/2} \leq\ C\ \underset{L\rightarrow\infty}{\limsup}\ \frac{1}{L^{n-\frac{\alpha p}{2}}} \int_{|\xi|\leq L} |\wh{u}(\xi)|^p d\xi <\infty.
      \eee
\end{proof}

In \cite{AgmonHormander}, the authors studied the Fourier asymptotics of measures supported in a smooth manifold and proved similar results:\\

\emph{Let $u$ be a tempered distribution such that $\wh{u}\in L^2_{loc}$ and}
$$\underset{L\rightarrow\infty}{\limsup}\frac{1}{L^{k}}\int_{|\xi|\leq L}|\wh{u}(\xi)|^2d\xi<\infty.$$
\emph{If the restriction of $u$ to an open subset $X$ of $\R^n$ is supported by a $C^1$-submanifold $M$ of codimension $k$, then it is an $L^2$-density $u_0dS$ on $M$ and}
$$\int_M|u_0|^2dS\leq C\ \underset{L\rightarrow\infty}{\limsup}\frac{1}{L^k}\int_{|\xi|\leq R} |\wh{u}(\xi)|^2d\xi,$$
\emph{where $C$ only depends on $n$.}


\section{$L^p-$asymptotic properties of fractal measures for $1\leq p \leq 2$}
\setcounter{equation}{0}
Let $\mu$ denote a fractal measure supported in an
$\alpha$-dimensional set $E\subset\R^n$ and $f\in L^q(d\mu)$
($1\leq q\leq\infty$). Suppose $1\leq p\leq 2$ dependent on $q$.
In this section, we obtain lower bounds for
  \bee
    \underset{L\rightarrow\infty}{\liminf}L^{-k}\int_{|\xi|\leq L}|\wh{fd\mu}(\xi)|^pd\xi,
  \eee
for positive $k=n-\alpha$ and prove generalized Hardy inequality for fractal measures on $R^n$ of dimension $0<\alpha<n$.\\

Consider the generalized Hardy inequality for discrete measures proved by the authors in \cite{Hudson}:
\begin{thm}\label{thmHudson3zerop2}\cite{Hudson}
Let $c_k$ be a sequence of complex numbers, $a_k$ be a sequence of
real numbers and $fd\mu_0$ denote the zero dimensional measure
$f(x) = \sum_1^{\infty}c_k\delta(x-a_k)$ where $\delta$ is the
usual Dirac measure at zero.
\begin{enumerate}
\item Let $a_1<a_2<...$ and  assume $\wh{fd\mu_0}=\sum
c_ke^{ia_kx}$ belongs to the class of almost periodic functions.
Then,
 \bee
  \sum_1^{\infty}\frac{|c_k|}{k}\leq C\ \underset{L\rightarrow\infty}{\lim} L^{-1}\int_{-L}^L
  |\wh{fd\mu_0}(x)|dx.
 \eee
\item  Let $a_k$ be a sequence of real numbers, not necessarily
increasing and $1<p\leq 2$. Assume that $u(x)=\wh{fd\mu_0(x)}$
converges to $\sum_1^{\infty}c_ke^{ia_kx}$ in the class of almost
periodic functions. Then
 \bee
  \sum_1^{\infty}\frac{|c_k|^p}{k^{2-p}}\leq \sum_1^{\infty} \frac{|c_k'|^p}{k^{2-p}}\leq C\ \lim L^{-1}\int_{-L}^L|u(x)|^pdx,
 \eee
where $c_k'$ is the nonincreasing rearrangement of the sequence
$|c_k|$.
\end{enumerate}
\end{thm}
The authors also proved generalized Hardy inequality for fractal measures $fd\mu$ on $\R^1$ of dimension $\alpha$ $(0<\alpha<1)$ in \cite{Hudson} by generalizing part (1) of the above theorem with additional hypothesis on $\mu$. To prove the same, they introduced $\alpha$-coherent sets in $\R$ ($0<\alpha<1$).
Given $x\in\R$ and a set $E\subset\R$, let $E_x=E\cap(-\infty,x]$.
Let $s=sup\{x:\CH_{\alpha}(E_x)<\infty\}$, $E^0=(E_s)^*$ where,
for a set $E$,
    $$E^*=\{x\in E: 2^{-\alpha}\leq\ol{D^{\alpha}}(\CH_{\alpha}|_E,x)\leq1\}.$$
The set $E\subset\R$ is $\alpha$-coherent ($0<\alpha<1$), if
there is a constant $C$ such that for all $x\leq s$,
    \bee
           \underset{\delta\rightarrow 0}{\limsup}\ |E_x^{0}(\delta)|\delta^{\alpha-1} \leq C\CH_{\alpha}(E_{x}^0),
    \eee
where $|E_x^{0}(\delta)|$ denotes the one dimensional Lebesgue
measure of the $\delta$-distance set $E_x^0(\delta)$ of $E_x^0$. The following was proved in \cite{Hudson}.

\begin{thm}\label{thmHudson1cohequasi}\cite{Hudson}
Suppose $0<\alpha<1$, $f\in L^{1}(d\CH_{\alpha})$ and
$\mu=\CH_{\alpha}|_E$ where $E$ is either $\alpha$-coherent or
quasi $\alpha$-regular. Then, there exists a non-zero finite
constant independent of $f$ such that
 \bee
  \int_E \frac{|f(x)|d\mu(x)}{\CH_{\alpha}(E_x^0)} \leq\ C\ \underset{L\rightarrow\infty}{\liminf} L^{\alpha-1} \int_{-L}^L |\wh{fd\mu}(x)|dx.
 \eee
\end{thm}
\begin{rem}
Examples in \cite{Hudson} show that there are quasi regular sets in $\R$ which are not $\alpha$-coherent and there are $\alpha$-coherent sets which are not quasi regular, for given $0<\alpha<1$.
\end{rem}
In this section, using the upper Minkowski content and finding a
continuous analogue of the arguments used in the proof of the Theorem \ref{thmHudson1cohequasi}, we prove an analogous version of part(2) of the Theorem \ref{thmHudson3zerop2} for $0<\alpha<n$, $n\geq 1$ and $1\leq p\leq 2$ with a slight modification in the hypothesis:\\

\begin{thm}\label{ThmHud1p2}
Let $E\subset\R^n$ be a compact set such that the $\alpha$-dimensional upper Minkowski content of $S$ is non-zero and bounded above by $\CH_{\alpha}(S)$ for any $S\subset E$ with $\CH_{\alpha}(S)\neq 0$ for some $0<\alpha<n$ and let $\mu=\CH_{\alpha}|_E$. Let $f\in L^p(d\mu)$ be a positive function, for $1\leq p\leq 2$. Then there exists a constant $C$ independent
of $f$ such that
  \be
     \int_E\frac{|f(x)|^p}{[\mu(E_x)]^{2-p}}d\mu(x) \leq\ C\ \underset{L\rightarrow\infty}{\liminf}
     \frac{1}{L^{n-\alpha}} \int_{|\xi|\leq L}|\wh{fd\mu}(\xi)|^pd\xi, \label{eqthmhud1p2}
  \ee
where $E_x=E\cap[(-\infty,x_1]\times...\times(-\infty,x_n]]$ for
$x=(x_1,...x_n)\in\R^n$.
\end{thm}
First we prove the following lemma:

\begin{lem}\label{lemHud1FT} Suppose $L>1$ and $0<\delta=r/L<1$ are given constants. Let $g_L\in
L^1(\R^n)$ and $S_{\delta}=\cup_{i=1}^{s} \Delta_i^{\delta}$ be
the union of disjoint cubes such that $0<|\Delta_i^{\delta}| <
\delta^n$. Then, there exists a non-zero finite constant $C_2$
independent of $g_L$, $s$, $\delta$ and $L$ such that
     \be
       \frac{\delta^{-n}}{P_{\delta}}\int_{S_{\delta}}|g_L(x)|dx \leq\ C_2\ \int_{\R^n}|\wh{g_L}(\xi)|d\xi, \label{eqlemHud1FT}
     \ee
where $P_{\delta}>1$ is a constant dependent on $\delta$.
\end{lem}

\begin{proof}
For all $i=1,...s$, construct $f_{i}\in L^2(\R^n)$ such that
  \Bea
    |\wh{f_{i}}(x)| &=&  \frac{\delta^{-n}}{P_{\delta}}\ \text{for}\ x\in \Delta_i^{\delta}\\
                     &=& 0\ \text{for}\ x\notin \Delta_i^{\delta}\\
    \wh{f_{i}}(x)g_L(x) &\geq& 0.
  \Eea
Since $|\Delta_i^{\delta}|\leq\delta^n$ and $P_{\delta}>1$, $\|\wh{f_{i}}\|_1\leq 1$ and hence for all $\xi$, $|f_{i}(\xi)|\leq 1$. Denote
$F_{0}\equiv0$. For all $i=1,...,s$, let
$$F_{i}(\xi) = \frac{4}{5}F_{i-1}(\xi)\exp(\frac{-1}{4s^2}|f_{i}(\xi)|) + \frac{f_{i}(\xi)}{20}$$
and denote $F\equiv F_{s}$. Since $|f_{i}(\xi)|\leq 1 $ for all
$i$, we have $|F_{1}(\xi)|\leq 1/4$. Note that for all $0\leq
t\leq 1$ and $s\geq 1$,
  \Bea
  \frac{4}{5}exp(\frac{-t}{4s^2})&\leq& 1-\frac{t}{5}\\
  \frac{1}{5}exp(\frac{-t}{4s^2})+\frac{t}{20}&\leq& \frac{1}{4}.
  \Eea
Since for all $\xi$, $|f_2(\xi)|\leq 1$, we have
   \bee
    |F_{2}(\xi)|\leq \frac{1}{5}exp(\frac{-|f_{2}(\xi)|}{4s^2}) + \frac{|f_{2}(\xi)|}{20} \leq \frac{1}{4}.
   \eee
Then by induction $\|F\|_{\infty}\leq 1/4$. By
construction, we have
   \be
   F(\xi)= \sum_{k=1}^{s-1} \bigg[ \frac{4^{s-k}f_{k}(\xi)}{5^{s-k}20}\exp\big(\frac{-1}{4s^2} \sum_{l=k+1}^{s}|f_{l}(\xi)|\big)\bigg] + \frac{f_{s}(\xi)}{20}.\label{eqlemFTredFj}
   \ee
Now consider $\wh{F}$,
   \Bea
    \wh{F}(x) &=& \sum_{k=1}^{s-1} \bigg[\frac{4^{s-k}f_{k}(\xi)}{5^{s-k}20} \exp\big(\frac{-1}{4s^2} \sum_{l=k+1}^{s}|f_{l}(\xi)|\big)\bigg]\wh{\ }(x) +
    \frac{\wh{f_{s}}(x)}{20}\\
    &=&\sum_{k=1}^{s-1} \bigg[\frac{4^{s-k}f_{k}(\xi)}{5^{s-k}20} \big(\exp(\frac{-1}{4s^2} \sum_{l=k+1}^{s}|f_{l}(\xi)|)-1\big)\bigg]\wh{\ }(x)\\
     &&\ + \sum_{k=1}^{s} \frac{4^{s-k}\wh{f_{k}}(x)}{5^{s-k}20}.
   \Eea
By the construction of $f_{i_0}'s$, for all $x\in
\Delta_{i_0}^{\delta}$, $|\wh{f_{i}}(x)|=0$ for all $i\neq i_0$ and
$\wh{f_{i}}(x)g_L(x)\geq 0$. Hence
   \Bea
   &&Re(\wh{F}(x)g_L(x))\\
   &&\leq \sum_{k=1}^{s-1} \bigg|\bigg[\frac{4^{s-k}f_{k}(\xi)}{5^{s-k}20} \big(\exp(\frac{-1}{4s^2} \sum_{l=k+1}^{s}|f_{l}(\xi)|)-1\big)\bigg]\wh{\ }(x)\bigg||g_L(x)|\\
   &&\ \ +\ \frac{4^{s-i_0}}{5^{s-i_0}20} \wh{f_{i_0}}(x)g_L(x) \\
   &&\leq \sum_{k=1}^{s-1} \frac{\|f_{k}\|_2}{20} \bigg{\|}\big(\exp(\frac{-1}{4s^2} \sum_{l=k+1}^{s}|f_{l}(\xi)|)-1\big)\wh{\ }\bigg{\|}_2|g_L(x)|\\
   &&\ \ +\ \frac{1}{20} \wh{f_{i_0}}(x)g_L(x).
   \Eea
That is, for $x\in \Delta_{i_0}^{\delta}$,
   \bea
   \nonumber &&Re(20\wh{F}(x)g_L(x))- \wh{f_{i_0}}(x)g_L(x)\\
    &&\ \ \leq  \sum_{k=1}^{s-1} \|f_{k}\|_2 \bigg{\|}\big(\exp(\frac{-1}{4s^2} \sum_{l=k+1}^{s}|f_{l}(\xi)|)-1\big)\wh{\ }\bigg{\|}_2|g_L(x)|. \label{eqlemFTredFjRe}
   \eea
Since for all $a>0$, $\bigg|\frac{\exp(-a)+1}{a}\bigg|\leq 1$ and for all $i$, $\|f_{i}\|_2\leq\delta^{-n/2}$ we have
   \Bea
   \sum_{k=1}^{s-1} \|f_{k}\|_2 \bigg{\|}\big(\exp(\frac{-1}{4s^2} \sum_{l=k+1}^{s}|f_{l}(\xi)|)-1\big)\wh{\ }\bigg{\|}_2
   &\leq& \sum_{k=1}^{s-1} \|f_{k}\|_2 \bigg(\sum_{l=k+1}^{s}\frac{\|f_{l}\|_2}{4s^2}\bigg)\\
   &\leq& \frac{\delta^{-n}}{8}.
   \Eea
Thus from (\ref{eqlemFTredFjRe}), for $x\in \Delta_{i_0}^{\delta}$
   \Bea
    &&\delta^{-n}|g_L(x)|=|\wh{f_{i_0}}(x)g_L(x)|\\
    &\leq& |\wh{f_{i_0}}(x)g_L(x) - Re(20\wh{F}(x)g_L(x))| + Re(20\wh{F}(x)g_L(x))\\
    &\leq& \frac{\delta^{-n}}{8}|g_L(x)| + Re(20\wh{F}(x)g_L(x)).
    \Eea
Thus for all $i$ and $x\in \Delta_{i_0}^{\delta}$, $0\leq\delta^{-n}|g_L(x)|\leq 40
Re(\wh{F}(x)g_L(x))$. Hence, for all $x$, $0\leq\delta^{-n}|g_L(x)|\leq 40
Re(\wh{F}(x)g_L(x))$ and
    \bea
   \nonumber \int_{S_{\delta}} \delta^{-n}|g_L(x)| dx &\leq&
   40 Re\bigg(\int_{\R^n}\wh{F}(x)g_L(x)dx\bigg)\\
   \nonumber &\leq& 40\int_{\R^n}|F(\xi)||\wh{g_L}(\xi)| d\xi,
  \eea
Also we have $\|F\|_{\infty}\leq 1/4$. Then,
  \bee
    \int_{S_{\delta}} \delta^{-n}|g_L(x)| dx \leq C_2 \int_{\R^n}|\wh{g_L}(\xi)|
    d\xi.
  \eee
Hence the proof.
\end{proof}

\textbf{Proof of Theorem \ref{ThmHud1p2}:}\\

Since $E$ is a bounded set, without loss of generality we assume
that $\tilde{m}>1$ is the smallest integer such that for all
$x=(x_1,...x_n)\in E$, $1\leq x_j\leq \tilde{m}$, $j=1,...n$. Fix
$0<\epsilon<1$ and $m=\tilde{m}+1$. Then $E(\epsilon)$, the
$\epsilon$-distance set of $E$ is contained in
$M=(0,m)\times...(0,m)$.\\

As in Theorem \ref{thm-cohe-str}, we approximate $fd\mu$ with a Schwartz function on a fine decomposition of $E(r/L)$, $r/L$-distance set of $E$ for very small $r/L$ depending on $\epsilon$. First, we consider a self-similar Cantor type fractal $\tilde{C}_{\epsilon}$ as in the proof of Theorem \ref{thmHudson1cohequasi} in \cite{Hudson} such that $C_{\epsilon}$ has small $\alpha$-Hausdorff measure and set $C_{\epsilon}$ the $n$-times cartesian product of $\tilde{C}_{\epsilon}$.
$$$$
Construct a self-similar Cantor-type set $C$ in
$[-2/\epsilon,-1/\epsilon]\times...[-2/\epsilon,-1/\epsilon]\subset\R^n$
satisfying open set condition with dilation factor $0<\eta<1$ such
that $N\eta^{\alpha}=1$ and $\CH_{\alpha}(C)=1$. (See Definition
\ref{defnSelfsimilar}.) Let $C_{\epsilon}$ denote the
$\epsilon$-dilated $C$ such that $C_\epsilon\subset
[-2,-1]\times..\times[-2,-1]=M_1$ and $\CH_{\alpha}(C_{\epsilon})
= \epsilon^{\alpha}\CH_{\alpha}(C) = \epsilon^{\alpha}$.
Denote $E'=E\cup C_{\epsilon}$. Thus for all $x\in E$, $\mu(E'_x)=\mu(E_x)+\CH_{\alpha}(C_\epsilon)$ and also $E'$ is such that the upper Minkowski content of $E'_x$ is nonzero and bounded above by their Hausdorff measure. Hence
  \bea
   \nonumber \int_E\frac{|f(x)|^p}{\mu(E_x)^{2-p}}d\mu(x) &=&
   \underset{\epsilon\rightarrow0}{\lim}
   \int_E\frac{|f(x)|^p}{(\mu(E_x) + (\eta^{-1}\epsilon)^{\alpha} + \epsilon)^{2-p}}d\mu(x)\\
   &\leq& \underset{\epsilon\rightarrow0}{\lim}
   \int_E\frac{|f(x)|^p}{(\mu(E'_x) + \epsilon)^{2-p}}d\mu(x) \label{eqlemAuxReductionepsilon}
  \eea
Now to approximate $fd\mu$ with a Schwartz function, we proceed as in the Theorem \ref{thm-cohe-str}.\\

Fix $\epsilon_1<\epsilon/2$. For each $k=(k_1,...k_n)$, ($0<k_j\in\Z$) denote $Q_k=\{x=(x_1,...x_n)\in M: (k_j-1)\epsilon_1<x_j \leq k_j\epsilon_1\}$. Let $\CQ_0$ denote the
finite collection of all such cubes whose intersection with $E$ that has non zero measure, that is, $\mu(Q_k)\neq0$. For every
$k=(k_1,...k_n)$, denote $x_k=((k_1-1)\epsilon_1,...(k_n-1)\epsilon_1)$, $E_{k}=E_{x_k}= E\cap \prod_{j=1}^n
(-\infty,(k_j-1)\epsilon_1]$ and $E'_{k}= E'_{x_k} = E'\cap \prod_{j=1}^n (-\infty,(k_j-1)\epsilon_1]$. Then for all $Q_k\in\CQ_0$ and $x\in Q_k$, $\mu(E'_k)\leq \mu(E'_x)$. Also for all $k$,
$\mu(E'_k)=\mu(E_k)+\CH_{\alpha}(C_{\epsilon})>0$. Since $E$ is
compact, $\CQ_0$ has finite disjoint collection of half open cubes.
Hence
       \bea
          \nonumber \int_E\frac{|f(x)|^p}{(\mu(E'_x)+ \epsilon)^{2-p}}d\mu(x) &=& \sum_{Q_k\in\CQ_0}\int_{Q_k}
          \frac{|f(x)|^p}{(\mu(E'_x)+ \epsilon)^{2-p}}d\mu(x)\\
          \nonumber &\leq& \sum_{Q_k\in\CQ_0}\int_{Q_k}
          \frac{|f(x)|^p}{(\mu(E'_k)+ \epsilon)^{2-p}}d\mu(x)\\
          \nonumber &\leq& \frac{C_p}{(\epsilon)^{2-p}} \sum_{Q_k\in\CQ_0} \int_{Q_k}\Bigg|f(x)- \frac{1}{(\mu(Q_k))^{p}} \int_{Q_k} f(y) d\mu(y)\bigg|^p d\mu(x)\\
          && + \sum_{Q_k\in\CQ_0} \frac{\mu(Q_k)^{1-p}}{(\mu(E'_k)+ \epsilon)^{2-p}}\bigg|\int_{Q_k}
          f(y) d\mu(y)\bigg|^p.\label{thmpRed1}
          \eea
Let $i_{\epsilon_1}=\inf_{Q\in\CQ_0}\mu(Q)$. Since infimum is taken
over cubes in $\CQ_0$, which is a finite collection and $\mu(Q)\neq
0$, we have $i_{\epsilon_1}>0$. Now, for
each $k$, there exists $\delta_k$ such that
   \bea
    \nonumber |(Q_k\cap E)(\delta)|\delta^{\alpha-n} &\leq& C_n\mu(Q_k\cap E) + C_ni_{\epsilon_1}\epsilon,\\
    &\leq& 2C_n\mu(Q_k\cap E)\ (\text{since}\ \epsilon<1),\label{thmp1}\\
   |E'_k(\delta)|\delta^{\alpha-n}&\leq& C_n\mu(E'_k) + C_n\epsilon, \label{thmp1coh}
    \eea
for all $\delta \leq\delta_k$. Let $\tilde{\delta}_1\leq \min_k\{\delta_k\}$.\\

Let $\phi$ be a positive Schwartz function such that
$\wh{\phi}(0)=1$, support of $\wh{\phi}$ is supported in the unit
ball and there exists $r_1>0$ such that
\be
\int_{A_{r_1}(0)}
\phi(x) dx = 1/2^{n+1},\label{thmp2}
\ee
where $A_{r_1}(0)=\{x=(x_1,...x_n)\in\R^n: -r_1<x_j\leq 0,\ \forall j\}$. Denote $\phi_L(x)=\phi(Lx)$ for all $L>0$. Fix $\delta_0\leq
\min\{\epsilon, \tilde{\delta}_1\}$, $r=n^{\frac{1}{2}}r_1$ and $L$ large such that
$r/L\leq \delta_0$. Then we have,
      \Bea
      \bigg|\int_{Q_k} f(y) d\mu(y)\bigg|^p &=& 2^{p(n+1)} \bigg|\int_{Q_k} \int_{A_{r_1}(0)} \phi(x) dx f(y) d\mu(y)\bigg|^p\\
      &=& 2^{p(n+1)}L^{np} \bigg|\int_{Q_k} \int_{A_{r_1/L}(y)} \phi_L(x-y) dx f(y) d\mu(y)\bigg|^p\\
      &=& 2^{p(n+1)}L^{np} \bigg|\int_{Q_kE_L} \int_{Q_k} \phi_L(x-y)f(y)
      d\mu(y)dx\bigg|^p.
      \Eea
where $Q_kE_L=\{x=(x_1,...x_n)\in M:\ \exists y=(y_1,...y_n)\in E,\ \text{such} \ \text{that}\ y_j-r_1/L< x_j\leq y_j\ \forall\ j\}$. Note that $|Q_kE_L|\leq |(Q_k\cap E)(r/L)|$, where $(Q_k\cap E)(r/L)$ denotes the $r/L$-distance set of $Q_k\cap E$ (since $r=n^{1/2}r_1$). Since $\phi$ and $f$ are positive, $\int_{Q_kE_L} \int_{Q_k} \phi_L(x-y)f(y)d\mu(y)dx \leq \int_{Q_kE_L}\phi_L*fd\mu(x)dx$. Thus
      \Bea
      &&\bigg|\int_{Q_k} f(y) d\mu(y)\bigg|^p\\
      &&\ \leq 2^{p(n+1)} L^{np} \bigg|\int_{Q_kE_L} \phi_L*fd\mu(x)dx\bigg|^p\\
      &&\ \leq 2^{p(n+1)}L^{np}(|Q_kE_L|)^{p-1} \int_{Q_kE_L} |\phi_L*fd\mu(x)|^pdx\\
      &&\ \leq 2^{p(n+1)}r^{(n-\alpha)(p-1)} L^{n+\alpha(p-1)}(|(Q_k\cap E)(r/L)|(r/L)^{(\alpha-n)(p-1)} \int_{Q_kE_L}
      |\phi_L*fd\mu(x)|^pdx.
      \Eea
By (\ref{thmp1}), there exists a constant $\tilde{C}$ independent of $f$, $\epsilon$, and $L$ such that
      \be
       \frac{1}{\mu(Q_k)^{p-1}}\bigg|\int_{Q_k} f(y) d\mu(y)\bigg|^p \leq \tilde{C}L^{n+\alpha(p-1)} \int_{Q_kE_L}
       |\phi_L*fd\mu(x)|^p dx. \label{thmpRed3}
      \ee
Let
\be
e_{\epsilon_1}= \sum_{Q_k\in\CQ_0}e_k = \sum_{Q_k\in\CQ_0} \int_{Q_k}\Bigg|f(x)- \frac{1}{(\mu(Q_k))^{p}} \int_{Q_k} f(y) d\mu(y)\bigg|^p d\mu(x).\label{eqerror}
\ee
Then from (\ref{thmpRed1}), (\ref{thmp1coh}) and (\ref{thmpRed3}), there exists a constant $\tilde{C_1}$ independent of $f$, $\epsilon$ and $L$ such that
    \bea
     \nonumber &&\int_E\frac{|f(x)|^p}{(\mu(E'_x)+\epsilon)^{2-p}} d\mu(x)\\
     \nonumber &&\ \ \leq C_p\epsilon^{p-2}e_{\epsilon_1} +  C_p\tilde{C}L^{n+\alpha(p-1)} \sum_{Q_k\in\CQ_0} \int_{Q_kE_L}
       \frac{|\phi_L*fd\mu(x)|^p}{(\mu(E'_k)+ \epsilon)^{2-p}}
       dx\\
     &&\ \ \leq C_p\epsilon^{p-2}e_{\epsilon_1} +  \tilde{C_1} L^{n(p-1)+\alpha} \sum_{Q_k\in\CQ_0} \int_{Q_kE_L}
       \frac{|\phi_L*fd\mu(x)|^p}{(|E'_k(r/L)|)^{2-p}}
       dx,\label{thmpRed2}
     \eea

For given $\epsilon_1$, let $g\in C_c^{\infty}(d\mu)$ be such that $\|f-g\|^p_{L^p(d\mu)}<\epsilon_1$. Then, as in the proof of Theorem \ref{thm-cohe-str}, we have
\Bea
 e_{\epsilon_1} &\leq& 2C_p^2\sum_{k}\int_{Q_k}|f(x)-g(x)|^pd\mu(x)\\
&&+C_p\sum_{k}\int_{Q_k}\bigg|g(x)-\frac{1}{\mu(Q_k)}\int_{Q_k} g(y)d\mu(y)\bigg|^pd\mu(x).
\Eea
Since $E=\cup_k(Q_k\cap E)$,
\bea
\nonumber e_{\epsilon_1}&\leq& 2C_p^2\|f-g\|^p_{L^p(d\mu)} +C_p\sum_{k}\int_{Q_k}\bigg|g(x)-\frac{1}{\mu(Q_k)}\int_{Q_k} g(y)d\mu(y)\bigg|^pd\mu(x)\\
&\leq& 2C_p\epsilon_1 +2\sum_{k}\int_{Q_k}\bigg|g(x)-\frac{1}{\mu(Q_k)}\int_{Q_k} g(y)d\mu(y)\bigg|^2d\mu(x).\label{Maxop1}
\eea
Since $g$ is compactly supported continuous function, $g$ is uniformly continuous and $$|g(x)-\frac{1}{\mu(Q_k)}\int_{Q_k} g(y)d\mu(y)|\rightarrow
0$$ uniformly in $x$ and $Q_k$ as $\mu(Q_k)\rightarrow 0$. As $\epsilon_1\rightarrow0$, we have $\mu(Q_k)\rightarrow 0$ for all $k$. Hence
\Bea
&&\sum_{k}\int_{Q_k}\bigg|g(x)-\frac{1}{\mu(Q_k)}\int_{Q_k} g(y)d\mu(y)\bigg|^pd\mu(x)\\
&\leq& \sum_{k}\mu(Q_k)\underset{Q_k\in\CQ_0}{\sup}\underset{x\in Q_k}{\sup}\bigg|g(x)-\frac{1}{\mu(Q_k)}\int_{Q_k} g(y)d\mu(y)\bigg|^p\\
&=& \mu(E)\underset{Q_k\in\CQ_0}{\sup}\underset{x\in Q_k}{\sup}\bigg|g(x)-\frac{1}{\mu(Q_k)}\int_{Q_k} g(y)d\mu(y)\bigg|^p,
\Eea
which goes to zero as $\epsilon_1$ goes to zero. Therefore,  from (\ref{Maxop1}), $e_{\epsilon_1}$ goes to zero as $\epsilon_1$ goes to zero.
$$$$
Since $r/L<\epsilon$, for each $k=(k_1,...k_n)$, $Q_kE_L$
intersects with at most $2^n-1$ other cubes $Q_{m}\cap E(r/L)$,
where $m=(m_1,...,m_n)$, $k_j-1\leq m_j\leq k_j$. Hence for each
$k$, $Q_kE_L$ is the union of $Q_k\cap E(r/L)$ and at most $2^n-1$
other sets $Q_m\cap E(r/L)$. Then for all such $m$,
$|E'_m(r/L)|\leq |E'_k(r/L)|$. Thus for each $k$, $$\int_{Q_k\cap
E(r/L)}\frac{|\phi_L*fd\mu(x)|^p}{|E'_k(r/L)|^{2-p}}dx$$ repeats
at most $2^n$ times. Let $\tilde{\CQ_0}$ denote the collection of
all $Q_k=\{x=(x_1,...x_n)\in M: (k_j-1)\epsilon_1<x_j \leq
k_j\epsilon_1\}$ where $k=(k_1,...k_n)$, ($0<k_j\in\Z$) such that
$|Q_k\cap E(r/L)|\neq 0$. Thus from (\ref{thmpRed2}) and
(\ref{thmpRed1}), there exists a constant $C_0$ independent of
$f,\epsilon$ and $L$ such that for all $r/L\leq \delta_0$,
  \bea
 \nonumber &&\int_E\frac{|f(x)|^p}{(\mu(E'_x)+ \epsilon)^{2-p}}d\mu(x)\\
 && \leq C_p\epsilon^{p-2}e_{\epsilon_1} + C_0L^{n(p-1)+\alpha} \sum_{Q_k\in\tilde{\CQ_0}} \int_{Q_k\cap E(r/L)}      \frac{|\phi_L*fd\mu(x)|^p}{|E'_k(r/L)|^{2-p}} dx,\label{eqthmpackfin1}
 \eea
where $e_{\epsilon_1}$ goes to zero as $\epsilon_1$ goes to zero.\\

Denote $\delta=r/L$. By the construction of $C_{\epsilon}$, for all $k$, $|C_{\epsilon}(\delta)|< |E'_k(\delta)|$. Also, by Lemma \ref{lemPM1}, $|C_{\epsilon}(\delta)|\geq C_n P(C_{\epsilon},\delta)\delta^n$. Denote $P_{\delta}= P(C_{\epsilon},\delta)>1$, the $\delta$-packing number of $C_{\epsilon}$. For $j=0,1,...J$, let
$\CS_j$ be the sub-collection of all $Q_k\in\tilde{\CQ_0}$ such that
$2^jP_{\delta}\delta^n\leq |E'_k(\delta)|<2^{j+1}P_{\delta}\delta^n$. We consider only
nonempty collections. Denote $g_L(x)=\phi_L*fd\mu(x)$. Then \Bea
 \sum_{Q_k\in\tilde{\CQ_0}} \int_{Q_k\cap E(\delta)}
       \frac{|g_L(x)|^p}{|E'_k(\delta)|^{2-p}} &=& \sum_j \sum_{Q_k\in\CS_j} \int_{Q_k\cap E(\delta)}
       \frac{|g_L(x)|^p}{|E'_k(\delta)|^{2-p}}\\
       &\leq& \sum_j (2^{j}P_{\delta}\delta^{n})^{p-2}\sum_{Q_k\in\CS_j} \int_{Q_k\cap E(\delta)}
       |g_L(x)|^pdx.
 \Eea
For each $j$, we can write $\cup_{Q_k\in\CS_j}Q_k\cap E(\delta) = S_j =
\cup_{i=1}^{s_j}\Delta_i^{\delta}$ as the finite disjoint union of
non-empty sets intersected with cubes of volume $\delta^n$, that
is, $0<|\Delta_i^{\delta}|\leq \delta^n$. Then \be
 \sum_{Q_k\in\tilde{CQ_0}} \int_{Q_k\cap E(\delta)}
       \frac{|g_L(x)|^p}{|E'_k(\delta)|^{2-p}} \leq \sum_j (2^{j}P_{\delta}\delta^{n})^{p-2}\int_{S_j}
       |g_L(x)|^pdx \label{eqthmpackfin2}
 \ee

 For every $j$, applying Lemma \ref{lemHud1FT}, we have \be
\frac{\delta^{-n}}{P_{\delta}} \int_{S_j}
       |g_L(x)|dx \leq C \int_{\R^n}
       |\widehat{g_L}(\xi)|d\xi. \label{eqthmpackfin3}
 \ee

We recall the following interpolation theorem due to Stein (See page 213 in \cite{BennettSharpley} for the proof):
\begin{thm}\label{thmSteinInterpolation} Let $(\CR,\mu)$ and
$(\CS, \nu)$ be totally $\sigma$-finite measure spaces and let $T$
be a linear operator defined on the $\mu$-simple functions on
$\CR$ taking values in the $\nu$-measurable functions on $\CS$.
Suppose that $u_i, v_i$ are positive weights on $\CR$ and $\CS$
respectively, and that $1\leq p_i,q_i\leq\infty$, $(i=0,1)$.
Suppose
  \bee
    \|(Tf)v_i\|_{q_i}\leq M_i\|fu_i\|_{p_i},\ \ (i=0,1)
  \eee
for all $\mu$-simple functions $f$. Let $0\leq\theta\leq1$ and
define
  \bee
  \frac{1}{p}=\frac{1-\theta}{p_0}+\frac{\theta}{p_1},\  \
  \frac{1}{q}=\frac{1-\theta}{q_0}+\frac{\theta}{q_1}
  \eee
and
  \bee
    u=u_0^{1-\theta}u_1^{\theta},\ \ v=v_0^{1-\theta}v_1^{\theta}.
  \eee
Then, if $p<\infty$, the operator $T$ has a unique extension to a
bound linear operator from $L^p_u$ into $L^q_v$ which satisfies
  \bee
    \|(Tf)v\|_q\leq\ M_0^{1-\theta}M_1^{\theta}\|fu\|_p,
  \eee
for all $f\in L^p_u$
\end{thm}

Let $v_0= \frac{\delta^{-n}}{P_{\delta}}\chi_{S_j}(x)$ and $v_1 = u_0 = u_1 = 1$, where $\chi_{S_j}$ denotes the characteristic function on $S_j$. Let $T$ be defined as $T(\psi)=\check{\psi}$, the inverse Fourier
transform of $\psi$. By (\ref{eqthmpackfin3}), we have for each
$j$ and $L$,
$$\|(Tg_L)v_0\|_1\leq C\|\widehat{g_L}\|_1$$
By Plancherel theorem, we have
$$\|(Tg_L)v_1\|_2\leq \|\widehat{g_L}\|_2$$
 Then applying the Theorem
\ref{thmSteinInterpolation}, for $1<p<2$, we have
   \be
     (\delta^{n}P_{\delta})^{p-2}\int_{S_j}|\phi_L*fd\mu(x)|^p dx\leq\ C' \int_{\R^n}|\wh{\phi_L*fd\mu}(\xi)|^pd\xi.\label{eqinterpolateLPRHS}
   \ee
where $C'$ is a non-zero finite constant independent of $f$. Using
(\ref{eqthmpackfin2}), (\ref{eqthmpackfin1}) and
(\ref{eqinterpolateLPRHS}), there exists a constant $C$
independent of $f$, $\epsilon$ and $L$ such that for very large
$L$
   \Bea
        &&\int_E \frac{|f(x)|^p}{(\mu(E'_x)+2\epsilon)^{2-p}} d\mu(x)\\
        &&\ \ \ \leq\ e_{\epsilon_1}\epsilon^{p-2} + C L^{n(p-1)+\alpha} \int_{\R^n}|\wh{\phi_L*fd\mu}(\xi)|^pd\xi.
   \Eea
Since $\phi$ is a Schwartz function such that $\wh{\phi}$ is
supported in unit ball, $\|\wh{\phi}\|_{\infty}\leq 1$ and
$\wh{\phi_L}(\xi)=L^{-n}\wh{\phi}(L^{-1}\xi)$,
   \bee
     \int_E\frac{|f(x)|^p}{(\mu(E'_x)+2\epsilon)^{2-p}}d\mu(x)\ \leq\ e_{\epsilon_1}\epsilon^{p-2} + C\ L^{\alpha+n(p-1)}
     \int_{|\xi|\leq L}\frac{|\wh{fd\mu}(\xi)|^p}{L^{np}}d\xi,
   \eee
for all $r/L\leq \delta_0$, where $\delta_0$ goes to zero as $\epsilon_1<\epsilon/2\rightarrow0$.
Hence letting $\epsilon_1$ to zero, we have
\bee
     \int_E\frac{|f(x)|^p}{[\mu(E'_x)+2\epsilon]^{2-p}}d\mu(x) \leq\ C\ \underset{L\rightarrow\infty}{\liminf} \frac{1}{L^{n-\alpha}} \int_{|\xi|\leq L}|\wh{fd\mu}(\xi)|^pd\xi.
  \eee
Letting $\epsilon$ go to zero,  using
(\ref{eqlemAuxReductionepsilon}), we have
   \bee
     \int_E\frac{|f(x)|^p}{[\mu(E_x)]^{2-p}}d\mu(x) \leq\ C\ \underset{L\rightarrow\infty}{\liminf} \frac{1}{L^{n-\alpha}} \int_{|\xi|\leq L}|\wh{fd\mu}(\xi)|^pd\xi.
  \eee
Hence the proof.\\



In \cite{Hudson}, the authors proved that the above result (\ref{eqthmhud1p2}) fails for $n=1$ without any restriction on the set $E$. It is not known for what optimal value of $k$ and optimal range of $p$, $\underset{L\rightarrow\infty}{\limsup} \frac{1}{L^{k}} \int_{|\xi|\leq L}|\wh{fd\mu}(\xi)|^pd\xi$ is non-zero and finite.

\section{Acknowledgements}
The author would like to express her sincere gratitude to her research supervisor, Prof. E. K. Narayanan for his guidance and
immense support. The author also wishes to thank Prof. Malabika Pramanik and Prof. Robert Strichartz for their valuable remarks. The author is grateful to Prof. Kaushal Verma for the financial assistance provided and UGC-CSIR for financial support. This work is a part of PhD dissertation and is supported in part by UGC Centre for Advanced Studies.


\end{document}